\documentclass[11pt,a4paper]{article}
\usepackage{hyperref}
\hypersetup{nesting=true,debug=true,naturalnames=true}
\usepackage{graphicx,upref}

\usepackage{amsmath,amssymb,amsthm,enumerate,mathtools,xfrac}
\usepackage{authblk}

\usepackage{xypic} %for diagrams
\usepackage{mathrsfs} %for script fonts
\usepackage{xspace}

\newcommand{\bs}{\backslash}

\newtheorem{thm}{Theorem}[section]
\newtheorem{lem}[thm]{Lemma}

\newtheorem{prop}[thm]{Proposition}
\newtheorem{cor}[thm]{Corollary}

\newtheorem{thmintro}{Theorem}

\theoremstyle{definition}
\newtheorem{defn}[thm]{Definition}
\newtheorem{ex}[thm]{Example}

\newtheorem{rem}[thm]{Remark}
\newtheorem*{notn}{Notation}

\newcommand{\ol}{\overline}

\newcommand{\defbold}{\textbf}

\newcommand{\inv}{^{-1}}

\newcommand{\CC}{\mathrm{C}}

\newcommand{\N}{\mathrm{N}}

\newcommand{\tdlc}{t.d.l.c.\@\xspace}
\newcommand{\tdlcsc}{t.d.l.c.s.c.\@\xspace}

\newcommand{\triv}{\{1\}}

\newcommand{\Homeo}{\mathrm{Homeo}}

\newcommand{\Aut}{\mathrm{Aut}}
\newcommand{\Inn}{\mathrm{Inn}}

\newcommand{\Comm}{\mathrm{Comm}}

\newcommand{\Res}{\mathrm{Res}}

\newcommand{\bC}{\mathbb{C}}

\newcommand{\bF}{\mathbb{F}}
\newcommand{\bN}{\mathbb{N}}

\newcommand{\bR}{\mathbb{R}}

\newcommand{\bZ}{\mathbb{Z}}

\newcommand{\mc}[1]{\mathcal{#1}}
\newcommand{\ms}[1]{\mathscr{#1}}

\makeindex

\begin{document}

\title{Distal actions on coset spaces in totally disconnected, locally compact groups}

\author{Colin D. Reid\thanks{The author is an ARC DECRA fellow.  Research supported in part by ARC Discovery Project DP120100996.}}
\affil{University of Newcastle, School of Mathematical and Physical Sciences, Callaghan, NSW 2308, Australia \\ colin@reidit.net}
%\email{colin@reidit.net}

\maketitle

\begin{abstract}
Let $G$ be a totally disconnected, locally compact (\tdlc) group and let $H$ be an equicontinuously (for example, compactly) generated group of automorphisms of $G$. We show that every distal action of $H$ on a coset space of $G$ is a SIN action, with the small invariant neighbourhoods arising from open $H$-invariant subgroups. We obtain a number of consequences for the structure of the collection of open subgroups of a \tdlc group.  For example, it follows that for every compactly generated subgroup $K$ of $G$, there is a compactly generated open subgroup $E$ of $G$ such that $K \le E$ and such that every open subgroup of $G$ containing a finite index subgroup of $K$ contains a finite index subgroup of $E$.  We also show that for a large class of closed subgroups $L$ of $G$ (including for instance all closed subgroups $L$ such that $L$ is an intersection of subnormal subgroups of open subgroups), every compactly generated open subgroup of $L$ can be realized as $L \cap O$ for an open subgroup of $G$.
\end{abstract}

\tableofcontents

%\addtocontents{toc}{\protect\setcounter{tocdepth}{1}}

\section{Introduction}

An important aspect of the theory of totally disconnected, locally compact (\tdlc) groups $G$ is the study of dynamics of (semi-)groups of automorphisms (or more generally, endomorphisms) of the group, including inner automorphisms.  The theory is most well-developed in the case of a $\bZ$-action by automorphisms, or iterates of a single endomorphism; see for example \cite{BaumgartnerWillis}, \cite{WillisEndo} and \cite{BGT}.  More generally, the methods for actions of $\bZ$ generalize in a well-behaved way to actions of compactly generated groups with an additional condition, called \emph{flatness}, that is automatically satisfied, for example, by any action of a finitely generated nilpotent group (see \cite{WillisFlat}, \cite{ShalomWillis} and \cite{ReidFlat}).

We regard a \tdlc group $G$ as a uniform space by equipping it with the right uniformity.  The result we obtain will apply to all subgroups $H$ of $\Aut(G)$ that are \defbold{equicontinuously generated}, meaning that meaning that $H$ has a symmetric generating set that is equicontinuous on $G$.  The main theorem of this article, which is an application of the main theorem of \cite{ReidEqui}, gives characterizations of the closed $H$-invariant subgroups $K$ of $G$ that contain the discrete residual in terms of the dynamics of the action of $H$ on $G/K$.

\begin{defn}
Given a topological group $G$ and a group of automorphisms $H$ of $G$, the \defbold{discrete residual} $\Res_G(H)$ is the intersection of all $H$-invariant open subgroups of $G$; if $H = \mathrm{Inn}(G)$ we define $\Res(G) = \Res_G(H)$.

Given a closed $H$-invariant subgroup $K$ of $G$, then $H$ acts by homeomorphisms on the quotient space $G/K$.  We say $H$ acts \defbold{N-distally} on $G/K$ if there is a neighbourhood of the trivial coset consisting of distal points for the action and \defbold{N-equicontinuously} on $G/K$ if there is a neighbourhood of the trivial coset consisting of equicontinuous points for the action.
\end{defn}

\begin{thmintro}[See \S\ref{sec:mainthm}]\label{intro:distal_SIN}
Let $G$ be a \tdlc group, let $H \le \Aut(G)$ be equicontinuously generated and let $K$ be a closed $H$-invariant subgroup of $G$.  Let $\mc{V}$ be the set of open $H$-invariant subgroups of $G$.  Then the following are equivalent:
\begin{enumerate}[(i)]
\item\label{distal_SIN:1} $H$ acts N-equicontinuously on $G/K$;
\item\label{distal_SIN:2} $H$ acts N-distally on $G/K$;
\item\label{distal_SIN:3} there is a neighbourhood $N/K$ of the trivial coset in $G/K$ such that given $x,y \in N/K$, if $\overline{Hx}$ is compact and $x \in \overline{Hy}$, then $y \in \overline{Hx}$;
\item\label{distal_SIN:4} $\{VK/K \mid V \in \mc{V}\}$ is a base of neighbourhoods of the trivial coset in $G/K$;
\item\label{distal_SIN:5} $K \ge \Res_G(H)$.
\end{enumerate}
\end{thmintro}

Since a compact set of automorphisms is always equicontinuous, Theorem~\ref{intro:distal_SIN} applies to all compactly generated automorphism groups of \tdlc groups, for example the group of inner automorphisms induced by a compactly generated subgroup of $G$.  Note also that $\Inn(G)$ is equicontinuous if and only if $G$ is a \defbold{SIN group}, meaning it has a base of conjugation-invariant identity neighbourhoods, and indeed a base of identity neighbourhoods consisting of open normal subgroups.  It is already known from examples that no analogue of Theorem~\ref{intro:distal_SIN} holds when $H$ is an arbitrary group of automorphisms of a \tdlc group; for instance, every nilpotent \tdlc group has distal action on itself, but there are nilpotent \tdlc groups that are not SIN groups.  Hence some sort of assumption that $H$ is generated by a `small' subset of $\Aut(G)$ is necessary.

There are consequences for the manner in which compactly generated groups can be embedded into general \tdlc groups.  In particular, every compactly generated subgroup has a \emph{reduced envelope} in the sense of \cite{ReidFlat}.

\begin{defn}
Let $G$ be a topological group and let $H$ be a subgroup of $G$.  A \defbold{reduced envelope} for $H$ in $G$ is an open subgroup $E$ of $G$ such that $H \le E$, and such that given any open subgroup $E_2$ of $G$ such that $|H:H \cap E_2|$ is finite, then $|E:E \cap E_2|$ is finite.
\end{defn}

\begin{thmintro}[See Theorem~\ref{thm:sr_envelope}]\label{intro:reduced_envelope}
Let $G$ be a \tdlc group and let $H$ be a subgroup of $G$ with equicontinuously generated action on $G$.  Then there is an open subgroup $E$ of $G$ with the following properties:
\begin{enumerate}[(i)]
\item\label{reduced_envelope:1} $E$ is a reduced envelope for $H$ in $G$;
\item\label{reduced_envelope:2} $E = H\Res_G(H)U$, where $U$ is a compact open subgroup of $G$;
\item\label{reduced_envelope:3} $E$ is equicontinuously generated, and if $H$ is compactly generated then so is $E$;
\item\label{reduced_envelope:4} $\Res_G(H)$ is the intersection of all open normal subgroups of $E$, so in particular, $\Res_G(H)$ is normal in $E$;
\item\label{reduced_envelope:5} $E/\Res_G(H)$ is a SIN group.
\end{enumerate}
\end{thmintro}

When they exist, reduced envelopes are unique up to finite index.  We can thus define the \defbold{envelope class} $E_G(H)$ of $H$ to be the set of open subgroups of $G$ commensurate to a reduced envelope for $H$, and say subgroups $H_1$ and $H_2$ are \defbold{envelope equivalent}, and write $H_1 \approx_G H_2$, if both $H_1$ and $H_2$ have reduced envelopes and $E_G(H_1) = E_G(H_2)$.

Envelope equivalence classes have representatives in any cocompact subgroup and in any subgroup of finite covolume, so are potentially a tool to understand lattices in \tdlc groups, especially cocompact lattices.

\begin{thmintro}[See \S\ref{sec:env_equivalence}]\label{intro:covol_envelopes}
Let $G$ be a \tdlc group and let $A$ be a subgroup (not necessarily closed).  Suppose that $\overline{A}$ is either cocompact or of finite covolume in $G$.  Then for every compactly generated subgroup $H_1$ of $G$, there is a subgroup $H_2$ of $A$ such that $H_1 \approx_G H_2$.  If $\overline{A}$ is cocompact in $G$, then $H_2$ can be chosen to be finitely generated.
\end{thmintro}

Envelope equivalence classes are also well-behaved under continuous homomorphisms.  Given a \tdlc group $G$, write $\mc{E}_G$ for the set of commensurability classes of compactly generated open subgroups of $G$.

\begin{thmintro}[See \S\ref{sec:env_equivalence}]\label{intro:embedding_envelopes}
Let $G$ and $H$ be \tdlc groups and let $\phi: G \rightarrow H$ be a continuous homomorphism.  Let $K_1$ and $K_2$ be compactly generated subgroups of $G$ such that $K_1 \approx_G K_2$.  Then $\phi(K_1) \approx_H \phi(K_2)$. 

In particular, there is a well-defined map
\[
\theta: \mc{E}_G \rightarrow \mc{E}_H; \; E_G(O) \mapsto E_H(\phi(O)).
\]
If $H = \phi(G)X$ for a compact set $X$, then $\theta$ is surjective.
\end{thmintro}

We also consider how open subgroups of a \tdlc group $G$ are related to open subgroups of closed subgroups of $G$.  Fix a \tdlc group $G$ and a closed subgroup $H$ of $G$.  An open subgroup $K$ of $H$ is \defbold{relatively separable} in $(G,H)$ if $K = H \cap O$ for an open subgroup $O$ of $G$; we say $H$ is \defbold{regionally relatively separable (RRS)} in $G$ if every compactly generated open subgroup of $H$ is relatively separable in $(G,H)$.

Not all closed subgroups are RRS subgroups (see Example~\ref{ex:different_classes}).  However, we can prove that many subgroups are RRS by introducing a smaller class, the RIO subgroups of $G$, and showing that this class has a number of closure properties.

\begin{defn}
A closed subgroup $H$ is a \defbold{RIO subgroup} of $G$ if every compactly generated open subgroup of $H$ is an intersection of open subgroups of $G$.  (The definition can be stated in several equivalent ways; see Definition~\ref{def:RIO}, Proposition~\ref{prop:RIO_equivalents} and Corollary~\ref{cor:RIO_liminf}.)
\end{defn}

\begin{thmintro}[See Theorem~\ref{thm:rio_closure_properties}]\label{thmintro:rio_closure_properties}
Let $G$ be a \tdlc group.  All of the following are RIO subgroups of $G$:
\begin{enumerate}[(i)]
\item\label{rio_closure_properties:1} any closed subgroup $H$ that acts distally on $G/H$ by translation (for example, any closed subnormal subgroup of $G$);
\item\label{rio_closure_properties:2} any RIO subgroup of a RIO subgroup of $G$;
\item\label{rio_closure_properties:3} any intersection of RIO subgroups of $G$;
\item\label{rio_closure_properties:4} the closure of any pointwise limit inferior of RIO subgroups of $G$.
\end{enumerate}
\end{thmintro}

\begin{thmintro}[{See \S\ref{sec:reg_sep}}]\label{thmintro:RIO_separation}
Let $G$ be a \tdlc group.  Then every RIO subgroup of $G$ is an RRS subgroup of $G$.
\end{thmintro}

We also give an application of RIO subgroups to the theory of elementary \tdlc second-countable groups in the sense of Wesolek (see \cite{Wesolek}): given a group $G$ of decomposition rank $\alpha$, we show that, subject to some obvious restrictions, every possible decomposition rank of a compactly generated closed subgroup of $G$ occurs as the rank of a compactly generated \emph{open} subgroup of $G$.  (See \S\ref{sec:xi}.)

Finally, in the case that $H$ is a compactly generated group of automorphisms, we use the results of the present paper together with the main theorem of \cite{ReidWesolek} to analyse the structure of $\Res_G(H)$.

\begin{thmintro}[See \S\ref{sec:normal_series}]\label{intro:normal_series}
Let $G$ be a \tdlc group, let $H \le \Aut(G)$ be compactly generated and let $R = \Res_G(H)$.  Then $H$ does not act N-distally on $R/S$ for any proper closed subgroup $S$ of $R$.  In particular, $\Res_R(H) = R$.

Moreover, there is a finite normal $H$-invariant series
\[
\triv = R_0 < R_1 < \dots < R_n = R,
\]
such that given $1 \le i \le n$, the factor $R_i/R_{i-1}$ satisfies at least one of the following:
\begin{enumerate}[(i)]
\item\label{normal_series:1} $R_i/R_{i-1}$ is infinitely generated and discrete, and $i < n$;
\item\label{normal_series:2} $[R,R_i] \le R_{i-1}$ and $i < n$;
\item\label{normal_series:3} $R_i/R_{i-1}$ is the largest compact normal $H$-invariant subgroup of $R/R_{i-1}$ on which $H$ acts ergodically;
\item\label{normal_series:4} $R_i/R_{i-1}$ is a nondiscrete chief factor of $R \rtimes H$, that is, there does not exist a closed normal $H$-invariant subgroup $S$ of $R$ such that $R_{i-1} < S < R_i$, and moreover $H$ does not leave invariant any proper compactly generated open subgroup of $R_i/R_{i-1}$.
\end{enumerate}
\end{thmintro}

\subsection*{Acknowledgement}  I thank Eli Glasner for bringing the article \cite{AGW} to my attention, which inspired a major revision of the present article.  I also thank Pierre-Emmanuel Caprace, Riddhi Shah and Phillip Wesolek for their very helpful comments and suggestions in the preparation of this article.

\begin{notn}
In this article, we require all topological groups to be Hausdorff.  Given a group $G$, a subset $X$ of $G$ and a set of automorphisms $Y$, we say that $X$ is \defbold{$Y$-invariant} if $y(X) = X$ for all $y \in Y$.  Given a group $G$ and subgroups $H$, $K$ and $L$ such that $L \le K$ and $K$ and $L$ are $H$-invariant (that is, $hKh\inv = K$ for all $h \in H$, and similarly for $L$), we understand the action of $H$ on $K/L$ to be given by conjugation unless otherwise specified.  Given subsets $X$ and $Y$ of a group $G$, we define $X\inv := \{x\inv \mid x \in X\}$, $XY := \{xy \mid x \in X, y \in Y\}$, and $X^n = \{x_1x_2 \dots x_n \mid x_1,x_2,\dots,x_n \in X\}$.  All group actions are left actions unless otherwise stated.
\end{notn}

\section{Structure of coset actions}

\subsection{Preliminaries on equicontinuity}

\begin{defn}
Let $X$ be a locally compact Hausdorff space with a compatible uniformity $\mc{U}$, let $x \in X$ and let $S \subseteq \Homeo(X)$.  The set $S$ is \defbold{equicontinuous at $x$} if for every entourage $E \in \mc{U}$, there is a neighbourhood $U$ of $x$ such that 
\[
\forall s \in S \; \forall y \in U: (s.x,s.y) \in E.
\]
We say $S$ is \defbold{equicontinuous} if $S$ is equicontinuous at every $x \in X$.
\end{defn}

In the definition of equicontinuity for actions on locally compact Hausdorff spaces (or more general uniformizable spaces), the choice of uniformity is significant, but on a compact Hausdorff space there is only one compatible uniformity.  Thus in the case that $X$ is locally compact Hausdorff and $\overline{Sx}$ is compact, as soon as $S$ is equicontinuous at $x$ with respect to some compatible uniformity, it is equicontinuous at $x$ with respect to every compatible uniformity.

Compact sets of homeomorphisms are equicontinuous.

\begin{lem}\label{lem:compact_equi}
Let $X$ be a locally compact Hausdorff space, let $S \subseteq \Homeo(X)$ and suppose that $S$ has compact closure in the compact-open topology of $\Homeo(X)$.  Then $S$ is equicontinuous with respect to every compatible uniformity.
\end{lem}

\begin{proof}
Since equicontinuity is inherited by subsets, we may assume $S = \overline{S}$, so $S$ is compact in the compact-open topology.  Let $E$ be a neighbourhood of the diagonal in $X \times X$.  Then for each $s \in S$, we see that $E$ contains $U_s \times U_s$ for some open neighbourhood $U_s$ of $s.x$.  Let $K_s$ be a compact neighbourhood of $x$ contained in $s\inv(U_s)$.  Then there is a basic neighbourhood $N_s$ of $s$ in $\Homeo(X)$ consisting of all homeomorphisms $h$ such that $hK_s \subseteq U_s$.  Given $h \in N_s$ and $y \in K_s$, we see that $h.x$ and $h.y$ are both contained in $U_s$, so $(h.x,h.y) \in E$.  Since $S$ is compact, there is a finite subset $\{s_1,\dots,s_n\}$ of $S$ such that $S \subseteq \bigcup^n_{i=1}N_{s_i}$.  Now let $U = \bigcap^n_{i=1}K_{s_i}$: we observe that whenever $y \in U$ and $s \in S$, we have $(s.x,s.y) \in E$.
\end{proof}

Equicontinuity is preserved by finite unions, and finite products of pointwise bounded equicontinuous families are pointwise bounded equicontinuous.

\begin{lem}[{\cite[Lemma~2.2]{ReidEqui}}]\label{lem:equi_product}
Let $X$ be a locally compact Hausdorff space equipped with some uniformity $\mc{U}$, let $S,T \subseteq \Homeo(G)$ and let $x \in X$.
\begin{enumerate}[(i)]
\item If $S$ and $T$ are equicontinuous at $x$, then so is $S \cup T$.
\item Suppose that $\overline{Tx}$ is compact, $T$ is equicontinuous at $x$, and $\overline{Tx}$ consists of equicontinuous points for $S$ on $X$.  Then $ST$ is equicontinuous at $x \in X$.  If in addition, $\overline{Sy}$ is compact for every $y \in \overline{Tx}$, then $\overline{STx}$ is compact.
\end{enumerate}
\end{lem}

\subsection{Equicontinuity for group automorphisms}

If $G$ is a locally compact group and $K$ is a closed subgroup of $G$, we have a canonical choice of compatible uniformity on the coset space $G/K$.

\begin{defn}
Let $G$ be a locally compact group and let $K$ be a closed subgroup of $G$.  The \defbold{right uniformity} on $G/K$ is the uniformity with basic entourages of the form
\[
E_U = \{(xK,yK) \in G/K \times G/K \mid xK \subseteq UyK\},
\]
where $U$ is an identity neighbourhood in $G$.  In particular if $K = \triv$, we have the right uniformity on $G$, with basic entourages of the form
\[
E_U = \{(x,y) \in G \times G \mid x \in Uy\}.
\]

Let $S \subseteq \Aut(G)$ such that $s(K) = K$ for all $s \in S$.  The action of $S$ on $G/K$ is \defbold{$1$-equicontinuous} if the trivial coset is an equicontinuous point for the action, \defbold{N-equicontinuous} if there is a neighbourhood of the trivial coset consisting of equicontinuous points, and \defbold{equicontinuous} if every coset is an equicontinuous point of the action, in all cases with respect to the right uniformity.  A group $H$ of automorphisms is \defbold{equicontinuously generated} if $H = \langle S \rangle$ where $S$ is equicontinuous on $G$ and $S = S\inv$.
\end{defn}

In the case that $S$ is a group, the $1$-equicontinuity condition is equivalent to $S$ having \defbold{small invariant neighbourhoods (SIN)} on $G/K$, that is, every neighbourhood of $K$ in $G/K$ contains a neighbourhood $O/K$ such that $s(O) = O$ for all $s \in S$.  However, a group can be equicontinuously generated without having a SIN action; consider for instance a compactly generated non-SIN group $G$ acting on itself by conjugation.

For a set of automorphisms acting on $G$ itself, equicontinuity at some point is equivalent to equicontinuity everywhere.

\begin{lem}\label{lem:equi_homogeneous}
Let $G$ be a locally compact group and let $S \subseteq \Aut(G)$.  Then $S$ is equicontinuous at some $x \in G$ if and only if $S$ is equicontinuous at every $x \in G$.
\end{lem}

\begin{proof}
Suppose that $S$ is equicontinuous at some $x \in G$.  Let $U$ be an identity neighbourhood and let $y \in G$.  Then there is a neighbourhood $V$ of $x$ such that for all $z \in V$ and all $s \in S$, we have $(s(x),s(z)) \in E_U$, that is, $s(x) \in Us(z)$.  We see that $V = V'x$ where $V'$ is an identity neighbourhood, so for all $v \in V'$ and $s \in S$ we have $s(x) \in Us(vx)$.  Using the fact that $s$ is an automorphism, we see that $s(x) \in Us(v)s(x)$, so $1 \in Us(v)$, and hence $s(y) \in Us(v)s(y) = Us(vy)$.  Thus $W = V'y$ is a neighbourhood of $y$ such that for all $z \in W$, we have $(s(y),s(z)) \in E_U$.  We conclude that $S$ is equicontinuous at every $y \in G$.
\end{proof}

\begin{cor}\label{cor:homogeneous_equi_product}
Let $G$ be a locally compact group and let $S,T \subseteq \Aut(G)$.  If $S$ and $T$ are equicontinuous on $G$, then so is $ST$.
\end{cor}

\begin{proof}
Let $x$ be the identity in $G$.  The set $\overline{Tx} = \{x\}$ is compact, so $ST$ is equicontinuous at $x$ by Lemma~\ref{lem:equi_product}, and hence by Lemma~\ref{lem:equi_homogeneous}, $S$ is equicontinuous on $G$.
\end{proof}

We thus have natural generalizations of `compactly generated' and `$\sigma$-compact'.  In particular, if one defines a \defbold{$\sigma$-equicontinuous} group of automorphisms to be a group of automorphisms arising as an ascending union of a sequence $(S_n)_{n \in \bN}$ of equicontinuous sets, then every equicontinuously generated group is $\sigma$-equicontinuous, as is every $\sigma$-compact group.

It is also useful to note that the equicontinuity property is inherited by the actions on all coset spaces where the action is well-defined.

\begin{lem}\label{lem:equi_quotient}
Let $G$ be a topological group, let $S \subseteq \Aut(G)$, let $H \le L \le K \le G$ be $S$-invariant closed subgroups and let $x \in K$.  Suppose $xH$ is an equicontinuous point of the action of $S$ on $G/H$.  Then $xL$ is an equicontinuous point of the action of $S$ on $K/L$.
\end{lem}

\begin{proof}
Let $U$ be an identity neighbourhood in $K$.  Then there is an identity neighbourhood $U_2$ in $G$ such that $U_2 \cap K \subseteq U$.  Since $xH$ is an equicontinuous point for $S$ on $G/H$, there is an identity neighbourhood $V$ in $G$ such that for all $y \in VxH$ and all $s \in S$, we have $s(x) \in U_2s(y)H$.  Consider now the set $W = (V \cap K)xL = (Vx \cap K)L$.  Certainly $W/L$ is a neighbourhood of $xL$ in $K/L$.  Given $y \in Vx \cap K$, we have $s(x) \in U_2s(y)H \cap K$; since $s(y) \in K$ and $H \le K$, we have
\[
U_2s(y)H \cap K = (U_2s(y) \cap K)H = (U_2 \cap K)s(y)H.
\]
In particular, $s(x) \in Us(y)L$.  Since $U$ was an arbitrary identity neighbourhood in $K$, we conclude that $xL$ is an equicontinuous point for the action of $S$ on $K/L$.
\end{proof}

A useful tool for understanding groups of automorphisms of a non-Archimedean group is the discrete residual, defined as follows.

\begin{defn}
Let $G$ be a topological group and let $H \le \Aut(G)$.  The \defbold{discrete residual} $\Res_G(H)$ of $H$ acting on $G$ is the intersection of all open $H$-invariant subgroups of $G$.
\end{defn}

If $K = HS$ is a group of automorphisms of $G$ such that $H$ is a subgroup and $S = S\inv$ is an equicontinuous set, then every open $H$-invariant subgroup of $G$ contains a $K$-invariant one.  In particular, the discrete residuals of the actions are the same.  Moreover $H$ is equicontinuously generated if and only if $K$ is, generalizing the well-known fact that a cocompact subgroup of a compactly generated locally compact group is compactly generated.

\begin{lem}\label{lem:equi_residual}
Let $G$ be a topological group and let $K$ be a group of automorphisms of $G$.  Suppose that $S = S\inv$, $K = HS$ where $H \le K$ and $S$ is equicontinuous on $G$.
\begin{enumerate}[(i)]
\item Every open $H$-invariant subgroup of $G$ contains a $K$-invariant one.  In particular, $\Res_G(H) = \Res_G(K)$.
\item The action of $H$ on $G$ is equicontinuously generated if and only if the action of $K$ is.
\end{enumerate}
\end{lem}

\begin{proof}
Without loss of generality we can assume $1 \in S$.

(i)
Let $O$ be an open $H$-invariant subgroup of $G$.  Then we have a $K$-invariant subgroup of the form
\[
O^* = \bigcap_{k \in K}k\inv(O) = \bigcap_{s \in S}s\inv(O).
\]
Since $S$ is equicontinuous, we see that $O^*$ is an open subgroup of $G$.  In particular, $O \ge \Res_G(K)$; given the choice of $O$, we conclude that $\Res_G(H) \ge \Res_G(K)$.  On the other hand, since every $K$-invariant open subgroup is in particular $H$-invariant, we must have $\Res_G(H) \le \Res_G(K)$, so in fact $\Res_G(H) = \Res_G(K)$.

(ii)
We adapt the proof of the analogous result for compact generation from \cite[Proposition~2.C.8]{CornulierDeLaHarpe}.
  
Suppose $T$ is an equicontinuous generating set for $H$ on $G$.  Then $S \cup T$ is a generating set for $K$ that is equicontinuous by Lemma~\ref{lem:equi_product}.

Conversely, suppose now that $T$ is an equicontinuous generating set for $K$ on $G$.  Without loss of generality, $1 \in T$ and $T = T\inv$.  Let $P = STS \cap H$.  We see that $P$ is equicontinuous on $G$ by Corollary~\ref{cor:homogeneous_equi_product}.  Clearly also $P = P\inv$.  It remains to show that $P$ generates $H$.  Let $h \in H$; then $h = t_1t_2\dots t_n$ for some $t_1,\dots,t_n \in T$.  Let $s_0 = 1$; since $K = HS$, for each $i \in \{1,\dots,n\}$, we can choose $s_i \in S$ such that $s_{i-1}t_is\inv_{i} \in H$ and hence $s_{i-1}t_is\inv_{i} \in P$.  Then
\[
\prod^n_{i=1}(s_{i-1}t_is\inv_{i}) = (\prod^n_{i=1}t_i)s\inv_n = hs\inv_n \in Hs\inv_n.
\]
The product $\prod^n_{i=1}(s_{i-1}t_is\inv_{i})$ is in $H$, so $s\inv_n$ is also in $H$, hence in $P$.  Thus we obtain an expression for $h$ as a product of elements of $P$, completing the proof that $H$ is equicontinuously generated on $G$.
\end{proof}

We can prove some additional properties of the action of $H$ on $G/K$ in the case that both $H$ and $K$ have $1$-equicontinuous action.

\begin{lem}\label{lem:coset:invariant_group}
Let $G$ be a topological group, let $K$ be a closed subgroup of $G$ and let $O$ be a nonempty open subset of $G$ such that $O = KOK$ and $O = O\inv$.  Suppose that $O \subseteq UK$ for some open subgroup $U$ of $G$.  Then there is an open subgroup $V$ of $U$ such that $VK = \langle O \rangle$.
\end{lem}

\begin{proof}
 Given $o \in O$, then $o = uk$ for some $u \in U$ and $k \in K$; it then follows $u = (uk)k\inv$ is an element of $Y := O \cap U$, so $O = YK$.  Indeed, $KYK = KO = O = YK$, so any product of sets of the form $A_1 A_2 \dots A_n K$ where $A_i \in \{Y,K\}$ can be reduced to the form $Y^mK$.  In particular, any such set is a subset of $UK$, so the semigroup $W$ generated by $Y \cup K$ is contained in $UK$.  Given $k \in K$, then $ko \in O$ for any $o \in O$, so $k = koo\inv \in OO\inv = O^2$.   Since $Y \subseteq O$, $K \subseteq O^2$, $O = YK$, and $O = O\inv$, we see that in fact $W$ is the group $\langle O \rangle$ generated by $O$.  In particular, $W$ has nonempty interior and is thus an open subgroup.  Since $K \le W \subseteq UK$, we have $W = VK$ where $V = U \cap W$.  Since $W$ is an open subgroup, so is $V$.
\end{proof}

\begin{defn}
Let $G$ be a group acting on a topological space $X$.  The \defbold{orbit closure relation} $R$ for the action is given by $(x,y) \in R$ if $y \in \overline{Gx}$.

An action of a group $G$ on a topological space $X$ is \defbold{minimal} if every orbit is dense.  More generally, the action has \defbold{minimal orbit closures} if $G$ acts minimally on $\overline{Gx}$ for every $x \in X$; equivalently, the orbit closure relation $R$ is symmetric.
\end{defn}

\begin{prop}\label{prop:SIN_tdlc}
Let $G$ be a \tdlc group, let $H \le \Aut(G)$ and let $K$ be a closed $H$-invariant subgroup of $G$.  Suppose $H$ and $K$ are both $1$-equicontinuous on $G/K$.
\begin{enumerate}[(i)]
\item
Let $\mc{W}$ be the set of $H$-invariant open subgroups $W$ of $G$ of the form $W = UK$, where $U$ is a compact open subgroup of $G$.  Then $\{W/K \mid W \in \mc{W}\}$ is a base of neighbourhoods of the trivial coset in $G/K$.  In particular, $K \ge \Res_G(H)$.
\item
$H$ has minimal orbit closures on $G/K$.
\end{enumerate}
\end{prop}

\begin{proof}
(i)
Let $\mc{U}$ be the set of compact open subgroups of $G$ and let $U \in \mc{U}$.  Since $K$ is $1$-equicontinuous on $G/K$, by applying Lemma~\ref{lem:coset:invariant_group}, we obtain an open subgroup $V$ of $G$ containing $K$ such that $V \subseteq UK$.  Since $H$ is $1$-equicontinuous on $G/K$, the $H$-invariant subgroup $W = \bigcap_{h \in H}h(V)$ is also open.  We have $K \le W \subseteq UK$, so that $W = (U \cap W)K$ as in the proof of Lemma~\ref{lem:coset:invariant_group}; in particular, $W \in \mc{W}$.  By Van Dantzig's theorem, $\{UK/K \mid U \in \mc{U}\}$ is a base of neighbourhoods of the trivial coset, so the same is true of the set $\{W/K \mid W \in \mc{W}\}$.  In particular, $K = \bigcap_{W \in \mc{W}}W \ge \Res_G(H)$.

(ii)
Let $x,y \in G$ and suppose $yK \subseteq \overline{H(x)K}$.  In other words, for every identity neighbourhood $U$ in $G$, there is $h \in H$ such that $h(x) \in yUK$.  Fix such an identity neighbourhood $U$.  By part (i) there is $W \in \mc{W}$ such that $W \subseteq UK$.  Now let $h \in H$ be such that $h(x) \in yW$; then $h\inv(y) \in xW \subseteq xUK$.  Since $U$ can be made arbitrarily small, we conclude that $xK \subseteq \overline{H(y)K}$.  Thus the condition $yK \subseteq \overline{H(x)K}$ defines a symmetric relation on pairs $(xK,yK)$ in $G/K$; in other words, $H$ has minimal orbit closures on $G/K$.
\end{proof}

\begin{cor}\label{cor:SIN_intersection}
Let $G$ be a SIN \tdlc group and let $K$ be a closed subgroup of $G$.  Then $K$ is the intersection of the open subgroups of $G$ that contain $K$ as a cocompact subgroup.
\end{cor}

\begin{proof}
Let $K$ act on $G$ by conjugation.  Then the action of $K$ on $G/K$ is equicontinuous by Lemmas~\ref{lem:equi_homogeneous} and \ref{lem:equi_quotient}.  Thus by Proposition~\ref{prop:SIN_tdlc}(i), we can express $K$ is an intersection of open subgroups of the form $UK$ where $U$ is a compact open subgroup of $G$; in particular, these contain $K$ as a cocompact subgroup.
\end{proof}

\subsection{Distality for group automorphisms}

Let $G$ be a topological group, let $H \le \Aut(G)$ and let $K$ be an $H$-invariant closed subgroup of $G$.

\begin{defn}
The action of $H$ on $G/K$ is \defbold{$1$-distal} if the trivial coset $K$ is a distal point of the action, and \defbold{N-distal} if the trivial coset is an N-distal point.
\end{defn}

We now have an array of possible properties of the action of $H$ on $G/K$ to consider:

\xymatrix{
\text{equicontinuous}  \ar@{=>}[d] \ar@{=>}[r] & \text{N-equicontinuous}  \ar@{=>}[d] \ar@{=>}[r] & \text{$1$-equicontinuous (SIN)} \ar@{=>}[d] \\
\text{distal}  \ar@{=>}[r] & \text{N-distal} \ar@{=>}[r] & \text{$1$-distal}
}

\

The implications shown are clear, but we do not know if any other implications hold in general; certainly distal does not imply equicontinuous.  The properties in each row are equivalent if $K$ is trivial, but possibly distinguished in general by complications arising from the action of $K$ on $G/K$.  We recall that by Lemma~\ref{lem:equi_quotient}, equicontinuous points are preserved on passing to quotient coset spaces.  In contrast there is no reason to expect distal points to pass to quotient coset spaces in general.

Distality implies a certain orbit closure property.

\begin{lem}[{\cite[Lemma~2.7]{ReidEqui}}]\label{lem:distal_orbit_closure}
Let $H$ be a group acting by homeomorphisms on a Hausdorff topological space $X$.  Let $x,y \in X$ be such that $\overline{Hx}$ is compact and consists of distal points for the action.  Then $H$ acts minimally on $\overline{Hx}$.  Indeed, we have $\overline{Hy} = \overline{Hx}$ whenever $y \in X$ is such that $\overline{Hx} \cap \overline{Hy} \neq \emptyset$.
\end{lem}

It is useful to have some equivalent versions of the condition ``$x$ is a distal fixed point'' (for example, when we have a $1$-distal action of automorphisms on a coset space $G/K$, and $x$ is the trivial coset).

\begin{lem}\label{lem:distal_equiv}
Let $G$ be a group acting by homeomorphisms on a Hausdorff topological space $X$ and let $x \in X$.  Then the following are equivalent:
\begin{enumerate}[(i)]
\item $x$ is a distal fixed point of the $G$-action;
\item For all nets $(g_i)_{i \in I}$ in $G$, and all points $y \in X \smallsetminus \{x\}$, then $(g_i(y))$ does not converge to $x$;
\item The intersection of the set $\{O \subseteq X \text{ open} \mid x \in O; \; \forall g \in G: gO = O\}$ is $\{x\}$.
\end{enumerate}
\end{lem}

\begin{proof}
Suppose $x$ is a distal fixed point of the $G$-action and let $(g_i)_{i\in I}$ be a net in $G$.  Suppose $g_i(y)$ converges to $x$ for some $y \in X$.  Then $(g_i(x),g_i(y))$ converges to $(x,x)$; since $x$ is a distal point we must have $y= x$.  Thus (i) implies (ii).

Suppose (ii) holds, let $y \in X \smallsetminus \{x\}$ and let $Y = \overline{\{g(y) \mid g \in G\}}$.  Then every point in $Y$ is the limit of some net $(g_i(y))$, where $(g_i)$ is a net in $G$.  Thus $x \not\in Y$.  We see that $O = X \smallsetminus Y$ is an open $G$-invariant subset of $X$ such that $x \in O$, but $y \not\in O$.  Since $y \in X \smallsetminus \{x\}$ was arbitrary, we conclude that (iii) holds.

Finally suppose that (iii) holds and let $y \in X \smallsetminus \{x\}$.  Then there exists an open $G$-invariant set $O$ that contains $x$ but not $y$.  In particular, no point $y \in X \smallsetminus \{x\}$ can be in the same $G$-orbit as $x$, so $Gx = \{x\}$, that is, $x$ is fixed by $G$.  Suppose $(g_i)$ is a net in $G$ such that $(g_i(y))$ converges to some limit $z$.  Then $(g_i(x))$ converges to $x$ (since $x$ is fixed), but $z$ lies in the closed $G$-invariant set $X \smallsetminus O$, which does not contain $x$.  Thus $(x,y)$ cannot be a proximal pair, showing that $x$ is a distal point for the action.
\end{proof}

With a sufficiently restricted action of $K$, we can ensure that a $1$-distal action of $H$ is in fact distal.

\begin{lem}\label{lem:1-distal_to_distal}
Let $G$ be a locally compact group, let $H \le \Aut(G)$ and let $K$ be a closed $H$-invariant subgroup of $G$.  Suppose that there is a collection $\mc{H}$ of $H$-invariant neighbourhoods of $K$ in $G/K$, each of which contains a $K$-invariant neighbourhood, and such that $\mc{H}$ has trivial intersection.  Then $H$ acts distally on $G/K$.
\end{lem}

\begin{proof}
Let $(h_i)$ be a net in $H$ and suppose $x,y,z\in G$ are such that $(h_i(x)K)$ and $(h_i(y)K)$ converge to $zK$.  Then there is a net $(k_i)$ in $K$ such that $(h_i(x)k_i)$ converges to $z$.  Consequently, the cosets
\[
k\inv_i(h_i(x)\inv h_i(y))K = k\inv_ih_i(x\inv y) K
\]
form a net that converges to the trivial coset in $G/K$.  In particular, $h_i(x\inv y)K$ is eventually contained in any given $K$-invariant neighbourhood of the trivial coset in $G/K$.  Since every element of $\mc{H}$ is $H$-invariant and contains a $K$-invariant neighbourhood of $K$, in fact $x\inv yK$ is contained in every element of $\mc{H}$.  By Lemma~\ref{lem:distal_equiv}, $\mc{H}$ has trivial intersection; thus $x\inv yK = K$, that is, $xK = yK$.  Thus the action of $H$ on $G/K$ is distal.
\end{proof}

We highlight the following special cases.

\begin{cor}\label{cor:res_distal}
Let $G$ be a locally compact group and let $H \le \Aut(G)$.
\begin{enumerate}[(i)]
\item The action of $H$ on $G/\Res_G(H)$ is distal.
\item Let $K$ be a closed $H$-invariant subgroup of $G$ such that $H$ acts $1$-distally on $G/K$ and $K$ is $1$-equicontinuous on $G/K$.  Then $H$ has distal action on $G/K$.
\end{enumerate}
\end{cor}

\begin{proof}
(i)
Let $R = \Res_G(H)$ and let
\[
\mc{H} = \{O/R \mid O \text{ is an open subgroup of } G, \forall h \in H: h(O) = O\}.
\]
By the definition of $R$, $\mc{H}$ has trivial intersection; moreover, every element of $\mc{H}$ is $R$-invariant.  Thus $H$ acts distally on $G/R$ by Lemma~\ref{lem:1-distal_to_distal}.

(ii)
Let
\[
\mc{H} = \{O/K \mid O \text{ is an open subset of } G, \; OK = O,\; \forall h \in H: h(O) = O\}.
\]
Since $H$ acts $1$-distally on $G/K$, the intersection of $\mc{H}$ is trivial.  Since $K$ is $1$-equicontinuous on $G/K$, any identity neighbourhood $O$ in $G$ that is invariant under right translation by $K$ (in other words, such that $O = OK$) also contains an identity neighbourhood that is invariant under left translation by $K$.  In particular, every element of $\mc{H}$ contains a $K$-invariant neighbourhood of the trivial coset.  Thus $H$ acts distally on $G/K$ by Lemma~\ref{lem:1-distal_to_distal}.
\end{proof}

\subsection{The stable residual property}\label{sec:sr}

Let $G$ be a \tdlc group and let $H \le \Aut(G)$.  In general $\Res_G(H)$ is not open, and so the set of open $H$-invariant subgroups has no minimal element.  However, there is often an open $H$-invariant subgroup $U$ such that $U/\Res_G(H)$ is compact; $U$ is then the smallest open $H$-invariant subgroup up to finite index.  In this section we define this property and then explore some consequences.

\begin{defn}
Let $G$ be a \tdlc group and let $H \le \Aut(G)$.  Say that $H$ has \defbold{stable residual (SR)} on $G$ if there exists an open $H$-invariant subgroup $U$ such that $\Res_G(H)$ is cocompact in $U$.
\end{defn}

If $H$ is equicontinuous on $G$, it is clear that $H$ has (SR) on $G$.  Later (Corollary~\ref{cor:equi_gen_res}), we will see that in fact every equicontinuously generated group of automorphisms has (SR).

The property (SR) has consequences for the normalizer of the discrete residual, as can be seen from the following lemma.

\begin{lem}\label{residual_normalizer}
Let $G$ be a \tdlc group, let $H \le \Aut(G)$, let $R = \Res_G(H)$ and let $x \in G$.  Suppose that $\{h(x)R \mid h \in H\}$ has compact closure in $G/R$.  Then $xRx\inv \le R$.
\end{lem}

\begin{proof}
Let $O$ be an open $H$-invariant subgroup of $G$.  Then $R \le O$, and since $\{h(x)R \mid h \in H\}$ has compact closure in $G/R$, the orbit $\{h(x) \mid h \in H\}$ is contained in the union of finitely many right cosets of $O$.  Consequently $O^* = \bigcap_{h \in H}h(x)\inv Oh(x)$ is an open subgroup of $G$; clearly $O^*$ is also $H$-invariant.  In particular, $R \le O^*$, so $R \le x\inv Ox$, that is, $xRx\inv \le O$.  Since $O$ was an arbitrary open $H$-invariant subgroup, in fact $xRx\inv \le R$.
\end{proof}

We now obtain a number of equivalent versions of (SR).

\begin{prop}\label{sr_equivalents}
Let $G$ be a \tdlc group, let $H \le \Aut(G)$, let $R = \Res_G(H)$ and let $\mc{W}$ be the set of $H$-invariant open subgroups $W$ of $G$ of the form $W = VR$, where $V$ is a compact open subgroup of $G$.  Then the following are equivalent:
\begin{enumerate}[(i)]
\item $H$ has (SR) on $G$.
\item $\mc{W}$ is nonempty.
\item $\{W/R \mid W \in \mc{W}\}$ is a base of neighbourhoods of the trivial coset in $G/R$, and $\N_G(R) \ge \langle \mc{W} \rangle$.
\item The action of $H$ on $G/R$ is $1$-equicontinuous.
\item $\N_G(R)$ is open and $H$ has equicontinuous action on $\N_G(R)/R$.
\end{enumerate}
\end{prop}

\begin{proof}
The implications (v) $\Rightarrow$ (iv), (iii) $\Rightarrow$ (ii) and (ii) $\Rightarrow$ (i) are all immediate.  It now suffices to show (iv) $\Rightarrow$ (iii) and (i) $\Rightarrow$ (v).

Suppose (iv) holds.  Since the trivial coset is an equicontinuous fixed point for $H$ on $G/R$, it has a compact $H$-invariant neighbourhood $O/R$.  By Lemma~\ref{residual_normalizer}, $O \cap O\inv \subseteq \N_G(R)$, so $\N_G(R)$ is open. Now consider the action of $H$ and $R$ on $\N_G(R)/R$.  Since $R$ acts trivially on $\N_G(R)/R$, we can apply Proposition~\ref{prop:SIN_tdlc}(i) to conclude that $\{W/R \mid W \in \mc{W}\}$ is a base of neighbourhoods of the trivial coset in $G/R$.  Lemma~\ref{residual_normalizer} ensures that $\N_G(R) \ge \langle \mc{W} \rangle$; thus (iii) follows.

Suppose (i) holds: say $U$ is an open $H$-invariant subgroup of $G$ such that $U/R$ is compact.  Then $U \le \N_G(R)$ by Lemma~\ref{residual_normalizer}, so $U/R$ is a compact group.  By the definition of $R$, the open $H$-invariant subgroups of $U/R$ have trivial intersection; by compactness, they therefore form a base of neighbourhoods of the identity.  Thus $H$ has $1$-equicontinuous action on $U/R$ and hence on $\N_G(R)/R$; since $\N_G(R)/R$ is a group, it follows by Lemma~\ref{lem:equi_homogeneous} that $H$ acts equicontinuously on $\N_G(R)/R$.  Thus (v) holds and the cycle of implications is complete.
\end{proof}

\subsection{A dynamical characterization of a class of coset spaces}\label{sec:mainthm}

We recall part of the main theorem of \cite{ReidEqui}; together with the results we have accumulated so far on coset space actions, this will lead to a proof of Theorem~\ref{intro:distal_SIN}.

\begin{thm}[See {\cite[Theorem~1.2]{ReidEqui}}]\label{thm:AGW_generalized}
Let $X$ be a locally compact zero-dimensional space, let $S$ be a family of homeomorphisms of $X$ with $1 \in S$ and $S = S^{-1}$, such that $S$ is equicontinuous with respect to some compatible uniformity, and let $G = \langle S \rangle$.  Let $U$ be an open subset of $X$, and write $U_0$ for the union of all compact $G$-invariant subsets of $U$.  Then the following are equivalent:
\begin{enumerate}[(i)]
\item For every compact open subset $V$ of $U$, there is a finite subset $F$ of $G$ such that $\bigcap_{g \in G}g(V) = \bigcap_{g \in F}g(V)$.
\item Given $x \in U$ and $y \in U_0$ such that $y \in \overline{Gx}$ is a compact subset of $U$, then $x \in \overline{Gy}$.
\end{enumerate}
\end{thm}

\begin{proof}[Proof of Theorem~\ref{intro:distal_SIN}]
It is clear that (\ref{distal_SIN:1}) implies (\ref{distal_SIN:2}).  By Lemma~\ref{lem:distal_orbit_closure} we see that (\ref{distal_SIN:2}) implies (\ref{distal_SIN:3}).

Now assume (\ref{distal_SIN:3}) holds.  Fix an equicontinuous generating set $S$ for $H$ such that $1 \in S$ and $S = S\inv$.  Then $S$ is equicontinuous on $G/K$ by Lemma~\ref{lem:equi_quotient}.  Considering the action of $H$ on $X = G/K$ and the neighbourhood $N/K$ of the trivial coset, we are in the situation of Theorem~\ref{thm:AGW_generalized}(ii).  Thus using Theorem~\ref{thm:AGW_generalized}(i), given any compact open subgroup $U$ of $G$ such that $U \subseteq N$, there is a finite subset $F$ of $H$ such that $O := \bigcap_{h \in H}h(UK)$ is $H$-invariant; note in particular that $O$ is an identity neighbourhood in $G$.  Let $W = \bigcap_{h \in F}h(U)$; then $W$ is an open subgroup of $G$, such that $wO = O$ for all $w \in W$.  In particular, the group $V := \{g \in G \mid gO = O\}$ is open.  Since $O$ is $H$-invariant, the group $V$ is also $H$-invariant, that is, $V \in \mc{V}$.  Since $O$ is an identity neighbourhood such that $VOK = O$, we have $VK \subseteq O$.  Since the initially chosen compact open subgroup $U$ of $G$ can be made arbitrarily small, by letting $U$ range over a base of identity neighbourhoods in $G$, we see that $\{VK/K \mid V \in \mc{V}\}$ is a base of neighbourhoods of the trivial coset in $G/K$.  Hence (\ref{distal_SIN:4}) holds.

Now suppose that (\ref{distal_SIN:4}) holds, that is, $\{VK/K \mid V \in \mc{V}\}$ is a base of neighbourhoods of the trivial coset in $G/K$.  Then in particular, the intersection $W = \bigcap_{V \in \mc{V}}V$ is contained in $K$.  By the definition of $\mc{V}$, we see that $W \ge \Res_G(H)$, so $K \ge \Res_G(H)$.  Thus (\ref{distal_SIN:4}) implies (\ref{distal_SIN:5}).

Finally, suppose that (\ref{distal_SIN:5}) holds.  Consider first the action of $H$ on $G/\Res_G(H)$: this action is distal, in particular N-distal, by Corollary~\ref{cor:res_distal}(i).  From the implication (\ref{distal_SIN:2}) $\Rightarrow$ (\ref{distal_SIN:4}) in the present theorem, indeed the action of $H$ on $G/\Res_G(H)$ is $1$-equicontinuous.  Proposition~\ref{sr_equivalents} then ensures that $H$ has (SR) on $G$, and hence has N-equicontinuous action on $G/\Res_G(H)$.  By Lemma~\ref{lem:equi_quotient}, the action of $H$ on $G/K$ is also N-equicontinuous, in particular N-distal.  Thus (\ref{distal_SIN:5}) implies (\ref{distal_SIN:1}) and the cycle of implications is complete.
\end{proof}

In the course of the above proof, we proved the following:

\begin{cor}\label{cor:equi_gen_res}
Let $G$ be a \tdlc group and let $H \le \Aut(G)$ be equicontinuously generated.  Then $H$ has (SR) on $G$.
\end{cor}

Let us note the special case of Theorem~\ref{intro:distal_SIN} when $K$ is $1$-equicontinuous on $G/K$ (for instance, if $K$ is a compact subgroup or if $K$ has open normalizer in $G$).  In this case we have more equivalent statements; the next theorem together with Corollary~\ref{cor:res_ergodic} will generalize \cite[Theorem~5.13]{ReidFlat}.

\begin{thm}\label{distal_SIN:compact}
Let $G$ be a \tdlc group, let $H \le \Aut(G)$ be equicontinuously generated and let $K$ be an $H$-invariant subgroup of $G$.  Suppose that $K$ is $1$-equicontinuous on $G/K$.  Let $\mc{V}$ be the set of open $H$-invariant subgroups of $G$ and let $\mc{W}$ be the set of open $H$-invariant subgroups of $G$ containing $K$ as a cocompact subgroup.  Then the action of $H$ on $G/K$ is one of: 

N-equicontinuous, $1$-equicontinuous, distal, N-distal, $1$-distal,

if and only if it has all of these properties.  Moreover, the following are equivalent:
\begin{enumerate}[(i)]
\item $H$ acts distally on $G/K$;
\item $H$ has minimal orbit closures on $G/K$;
\item $\{VK/K \mid V \in \mc{V}\}$ is a base of neighbourhoods of the trivial coset in $G/K$;
\item $\{W/K \mid W \in \mc{W}\}$ is a base of neighbourhoods of the trivial coset in $G/K$;
\item $K$ is an intersection of open $H$-invariant subgroups;
\item $K \ge \Res_G(H)$.
\end{enumerate}
\end{thm}

\begin{proof}
By Corollary~\ref{cor:res_distal} it is equivalent for the action of $H$ on $G/K$ to be distal, N-distal or $1$-distal.  Moreover, N-distal is equivalent to N-equicontinuous by Theorem~\ref{intro:distal_SIN}, N-equicontinuous implies $1$-equicontinuous, and $1$-equicontinuous implies $1$-distal.  The properties listed for the action of $H$ on $G/K$ are therefore all equivalent.

By Theorem~\ref{intro:distal_SIN}, (iii) and (vi) are each equivalent to $H$ having N-distal action.  By Proposition~\ref{prop:SIN_tdlc}(ii), (i) implies (ii); on the other hand it is clear that (ii) implies that the action is $1$-distal.  Thus (i), (ii), (iii) and (vi) are equivalent.

Assuming (iii), then (iv) follows by Proposition~\ref{prop:SIN_tdlc}(i).  The implications (iv) $\Rightarrow$ (v) $\Rightarrow$ (vi) are clear.  This completes proof that all the conditions are equivalent.
\end{proof}

Given a \tdlc group $G$ and a group $H$ acting distally on $G$, we observe an interesting closure property of equicontinuous subsets of $H$.  This generalizes \cite[Corollary~1.9]{ReidFlat}, which in turn generalized several previously known sufficient conditions for a \tdlc group $G$ to be a SIN group (in other words, for $\Inn(G)$ to be equicontinuous on $G$).

\begin{cor}\label{cor:distal_to_equi:auto}
Let $G$ be a \tdlc group and let $H \le \Aut(G)$.  Suppose that $H$ acts distally on $G$.  Then for every equicontinuous symmetric subset $S$ of $H$, the subgroup $\langle S \rangle$ generated by $S$ is itself equicontinuous.
\end{cor}

\begin{proof}
Given an equicontinuous subset $S$ of $H$, then $L = \langle S \rangle$ is equicontinuously generated and acts distally on $G$.  Hence by Theorem~\ref{intro:distal_SIN}, $L$ is N-equicontinuous on $G$; indeed, $L$ is equicontinuous by Lemma~\ref{lem:equi_homogeneous}.
\end{proof}

\section{Reduced envelopes}\label{sec:envelope}

\subsection{First results}\label{sec:envelope_existence}

Suppose that $H$ is a subgroup of $G$, acting by conjugation.  There is a close relationship between the set of open subgroups normalized by $H$, and the set of open subgroups containing $H$.  In particular, if $H$ has (SR) on $G$, then there is also an open subgroup containing $H$ that is the smallest such up to finite index.

\begin{defn}
Let $G$ be a \tdlc group and let $H \le G$.  An \defbold{envelope} for $H$ in $G$ is an open subgroup $E$ of $G$ such that $H \le E$.  The envelope is \defbold{reduced} if, given any open subgroup such that $|H:H \cap E_2| < \infty$, then $|E:E \cap E_2|<\infty$.
\end{defn}

The following is a sufficient condition for a group to be a reduced envelope.

\begin{lem}\label{lem:reduced_criterion}
Let $G$ be a \tdlc group and let $H \le G$.  Suppose that there is an open subgroup of $G$ of the form $E = \overline{H\Res_G(H)X}$, where $X$ is compact.  Then $E$ is a reduced envelope for $H$ in $G$.
\end{lem}

\begin{proof}
Let $E_2$ be an open subgroup of $G$, let $H_2 = H \cap E_2$ and suppose that $|H:H_2| < \infty$.  Then $H = H_2Y$ where $Y$ is finite, so $\Res_G(H) = \Res_G(H_2)$ by Lemma~\ref{lem:equi_residual}.  In particular, $\Res_G(H) \le E_2$.  Thus $E_2$ contains the finite index subgroup $\overline{H_2\Res_G(H)}$ of $\overline{H\Res_G(H)}$; the latter is cocompact in $E$ by hypothesis.  Since $E_2$ is also open, it follows that $|E:E \cap E_2|<\infty$.  Thus $E$ is a reduced envelope.
\end{proof}

If $H$ has a reduced envelope, recall that the \defbold{envelope class} $E_G(H)$ is the set of open subgroups of $G$ commensurate to a reduced envelope for $H$.  The discrete residual of $H$ on $G$ can be reinterpreted as the discrete residual of $E \in E_G(H)$.  The size of generating set needed for $E \in E_G(H)$ is controlled by a generating set for $H$.

\begin{lem}\label{lem:reduced_envelope:residual}
Let $G$ be a \tdlc group and let $H \le G$.  Suppose $H$ has a reduced envelope and let $E \in E_G(H)$.  
\begin{enumerate}[(i)]
\item Let $O$ be an open $H$-invariant subgroup of $G$.  Then there exists an open normal subgroup $O^*$ of $E$ such that $O^* \le O$.  Thus $\Res(E) = \Res_G(H)$; in particular, $\Res(E)$ does not depend on the choice of $E$, and the normalizer of $\Res_G(H)$ is open in $G$.
\item If $H$ is equicontinuously generated as a group of automorphisms of $G$, then so is $E$, and if $H$ is compactly generated, then so is $E$.
\end{enumerate}
\end{lem}

\begin{proof}
(i)
Let $N = \N_G(O)$ and let $E_2$ be a reduced envelope for $H$.  We see that $N$ is an open subgroup of $G$ containing $H$; by the reduced envelope property, $|E_2:E_2\cap N|$ is finite, so $|E:E \cap N|$ is finite.  It follows that $O_2 := \bigcap_{e \in E}eOe\inv$ is a finite intersection of conjugates of $O$, so $O_2$ is open, and hence $O^* = O_2 \cap E$ is open; by construction, $O^*$ is a normal subgroup of $E$.  We have shown that every $H$-invariant subgroup contains an open normal subgroup of $E$.  Conversely, since $E$ contains a finite index subgroup of $H$, the same argument shows that every open normal subgroup of $E$ contains an $H$-invariant open subgroup.  Thus $\Res(E) = \Res_G(H)$.  The remaining conclusions are clear.

(ii)
Let $U$ be a compact open subgroup of $E$.  Then $E_2 = \langle H, U \rangle$ is an open subgroup of $G$ containing $H$, so $E_2$ has finite index in $E$.  Thus $E = \langle H \cup U \cup X \rangle$ where $X$ is finite.  It follows that if $H$ is compactly generated, so is $E$.  Similarly, since the compact set $U \cup X$ is necessarily equicontinuous on $G$, if $H$ is equicontinuously generated on $G$ then so is $E$.
\end{proof}

We can now characterize the subgroups with (SR) on $G$ in terms of reduced envelopes.

\begin{lem}\label{lem:uniform_criterion}
Let $G$ be a \tdlc group and let $H \le G$.  Then the following are equivalent:
\begin{enumerate}[(i)]
\item $H$ has (SR) on $G$;
\item $H$ has a reduced envelope $E$ such that $E/\Res(E)$ is a SIN group.
\end{enumerate}
Moreover, if (ii) holds, we can take $E$ to be of the form $E = H\Res_G(H)U$ for a compact open subgroup $U$ of $G$.
\end{lem}

\begin{proof}
Let $R = \Res_G(H)$.

Suppose $H$ has (SR) on $G$.  Let $\mc{W}$ be the set of $H$-invariant open subgroups $W$ of $G$ of the form $W = UR$, where $U$ is a compact open subgroup of $G$.  Then by Proposition~\ref{sr_equivalents}, $\{W/R \mid W \in \mc{W}\}$ is a base of neighbourhoods of the trivial coset in $G/R$.  In particular, $H$ has an envelope of the form $E = HW$ for $W \in \mc{W}$; we can write $E$ as $E = HRU$ for $U$ a compact open subgroup of $G$, so $E$ is a reduced envelope by Lemma~\ref{lem:reduced_criterion}.  By Lemma~\ref{lem:reduced_envelope:residual}, we have $R = \Res(E)$ and every open $H$-invariant subgroup of $E$ contains an open $E$-invariant subgroup; in particular, every $W \in \mc{W}$ contains an open normal subgroup of $E$.  We conclude that $E/\Res(E)$ is a SIN group, so (i) implies (ii).

Conversely, suppose $H$ has a reduced envelope $E$ such that $E/\Res(E)$ is a SIN group.  Then $R=\Res(E)$ by Lemma~\ref{lem:reduced_envelope:residual}, so $\N_G(R)$ is open in $G$, and $E/R$ is a SIN group.  In particular, $H$ is equicontinuous on $E/R$, since $H \le E$, and hence $H$ has (SR) on $G$.  Thus (ii) implies (i). 
\end{proof}

We summarize the properties we have obtained for reduced envelopes of subgroups $H$ of a \tdlc group $G$, such that $H$ has (SR) on $G$.  This will immediately imply Theorem~\ref{intro:reduced_envelope}.

\begin{thm}\label{thm:sr_envelope}
Let $G$ be a \tdlc group and let $H \le G$.  If $H$ is compactly generated, or more generally, if $H$ has equicontinuously generated action on $G$, then $H$ has (SR) on $G$.

Now suppose that $H$ has (SR) on $G$ and let $K = \overline{H\Res_G(H)}$.  Then there is a reduced envelope $E$ for $H$ in $G$ of the form $E = KU$, where $U$ is a compact open subgroup of $G$.  Moreover the following holds for every closed subgroup $L$ of $G$ such that $K \le L$ and $|L:L \cap E| < \infty$:
\begin{enumerate}[(i)]
\item\label{sr_envelope:1} $L$ is compactly or equicontinuously generated on $G$ respectively if and only if $H$ is;
\item\label{sr_envelope:2} We have $\Res_G(H) = \Res(E) =\Res_G(L)$;
\item\label{sr_envelope:3} $L/\Res_G(H)$ is a SIN group;
\item\label{sr_envelope:4} Letting $\mc{W}$ be the set of open subgroups of $G$ containing $H$, then $\mc{W}$ has intersection $K$, and $\{W/K \mid W \in \mc{W}\}$ is a base of neighbourhoods for the trivial coset in $G/K$.
\end{enumerate}
\end{thm}

\begin{proof}
If $H$ has equicontinuously generated action on $G$, then $H$ has (SR) on $G$ by Corollary~\ref{cor:equi_gen_res}.  From now on we assume $H$ has (SR) on $G$.

Let $R = \Res_G(H)$.  By Lemma~\ref{lem:reduced_envelope:residual}, any reduced envelope $E$ of $H$ must satisfy $R = \Res(E)$, and hence $K \le E$.  Lemma~\ref{lem:uniform_criterion} shows that there is a reduced envelope $E$ for $H$ of the form $E = KU$, where $U$ is a compact open subgroup of $G$, and moreover that $E/R$ is a SIN group.  Part (\ref{sr_envelope:1}) is satisfied in the case $L \in E_G(H)$ by Lemma~\ref{lem:reduced_envelope:residual}, and then follows in the generality stated from Lemma~\ref{lem:equi_residual} and the fact that $K$ is cocompact in a reduced envelope.  We see that $L$ is contained in a reduced envelope $E_2$ for $H$, so $\Res_G(L) \le \Res(E_2) = R$; on the other hand, since $H \le L$, we have $R \le \Res_G(L)$.  Thus (\ref{sr_envelope:2}) holds.  Part (\ref{sr_envelope:3}) follows the fact that $E/R$ is a SIN group.

Given an open subgroup $U$ of $G$ that contains $H$, then $U$ is also $H$-invariant, so $R \le U$, and hence $\overline{HR} \le U$.  Recall that $R = \Res(E)$ by Lemma~\ref{lem:reduced_envelope:residual} and $E/R$ is a SIN group by Lemma~\ref{lem:uniform_criterion}.  By Corollary~\ref{cor:SIN_intersection}, $\overline{HR}/R$ is an intersection of open subgroups of $E/R$, and hence $\overline{HR}$ is an intersection of open subgroups of $G$.  Indeed, Proposition~\ref{prop:SIN_tdlc}(i) ensures that $\{W/\overline{HR} \mid W \in \mc{W}\}$ is a base of neighbourhoods for the trivial coset in $G/\overline{HR}$.  Thus (\ref{sr_envelope:4}) holds.
\end{proof}

Now let $G$ be a \tdlc group and let $H \le \Aut(G)$ be equicontinuously generated.  Using Theorem~\ref{intro:distal_SIN}, we see that given any $K < \Res_G(H)$, then the action of $H$ on $\Res_G(H)/K$ is not N-distal.

\begin{prop}\label{prop:residual_distal}
Let $G$ be a \tdlc group, let $H \le \Aut(G)$ be equicontinuously generated and let $R = \Res_G(H)$.  Then $H$ does not act N-distally on $R/S$ for any proper closed subgroup $S$ of $R$.  In particular, $\Res_R(H) = R$.
\end{prop}

\begin{proof}
By Theorem~\ref{intro:distal_SIN}, to show that $H$ cannot act N-distally on $R/S$ for any proper closed subgroup $S$ of $R$, it is sufficient (and clearly also necessary) to show that $H$ does not preserve any proper open subgroup of $R$.

Let $S$ be a proper open subgroup of $R$ and suppose $S$ is $H$-invariant.  Consider the action of $H$ on $G/S$.  Suppose that $(xS,yS)$ is a proximal pair for the action of $H$ on $G/S$.  By Corollary~\ref{cor:res_distal}, $H$ acts distally on $G/R$, so $xR = yR$, in other words $y =xr$ for some $r \in R$.  Let $(h_i)$ be a net in $H$ and let $t \in G$ be such that $h_i(x)S$ and $h_i(y)S$ converge to $tS$.  In particular, $h_i(y)s_i$ converges to $t$ for a net $(s_i)$ in $S$.  Now
\[
(h_i(y)s_i)\inv h_i(x)S = s\inv_i h_i(r\inv x\inv) h_i(x)S = s\inv_i h_i(r\inv)S
\]
is a net that converges to the trivial coset in $G/S$.  This net, however, is confined to the discrete space $R/S$, so in order to converge it must be eventually constant.  Thus $s\inv_i h_i(r\inv) \in S$ for sufficiently large $i$ and hence $r\inv \in S$, so that in fact $xS=yS$.  Thus the action of $H$ on $G/S$ is in fact distal.  Theorem~\ref{intro:distal_SIN} now implies that $S \ge R$, a contradiction.  We see then that there are no proper open $H$-invariant subgroups of $R$, as required.
\end{proof}

\begin{cor}\label{cor:res_ergodic}
Let $G$ be a \tdlc group and let $H \le \Aut(G)$ be equicontinuously generated.  If $\Res_G(H)$ is compact, then $H$ acts ergodically on $\Res_G(H)$.
\end{cor}

\begin{proof}
Follows from Proposition~\ref{prop:residual_distal} and \cite[Proposition~2.1]{JawDistal}.
\end{proof}

\subsection{Envelope equivalence}\label{sec:env_equivalence}

\begin{defn}
Let $G$ be a \tdlc group.  Say $H_1,H_2 \le G$ are \defbold{envelope equivalent} in $G$, and write $H_1 \approx_G H_2$, if $H_1$ and $H_2$ both have reduced envelopes and $E_G(H_1) = E_G(H_2)$.
\end{defn}

Let us now note some sufficient conditions for two compactly generated subgroups $H_1$ and $H_2$ of a group $G$ to be envelope equivalent.

\begin{lem}\label{lem:envelope_equiv}
Let $G$ be a \tdlc group and let $H$ be a compactly generated subgroup of $G$.
\begin{enumerate}[(i)]
\item Let $H_2$ be a subgroup of $G$ commensurate with $H$.  Then $H_2 \approx_G H$.
\item Let $E$ be an open subgroup of $G$.  Then $E \approx_G H$ if and only if $E \in E_G(H)$.
\item Let $K$ be a subgroup of $H$, such that $\overline{K}$ is either cocompact or of finite covolume in $\overline{H}$.  Then $K \approx_G H$.
\end{enumerate}
\end{lem}

\begin{proof}
(i)
Let $E$ be a reduced envelope for $H$ and let $E_2$ be a reduced envelope for $H_2$.  Then $|H:H \cap E_2| < \infty$, so $|E: E \cap E_2|< \infty$; similarly, $|H_2:H_2 \cap E| < \infty$, so $|E_2: E \cap E_2|< \infty$.

(ii)
Every open subgroup $E$ of $G$ is clearly its own reduced envelope in $G$, so $E_G(E)$ is just the commensurability class of $E$ amongst open subgroups of $G$.  Thus $E_G(E) = E_G(H)$ if and only if $E \in E_G(H)$.

(iii)
Since open subgroups are closed, it is easily seen that $E_G(H) = E_G(\overline{H})$ and similarly for $K$.  So we may assume $H$ and $K$ are closed subgroups of $G$.  Let $E$ be a reduced envelope for $H$.

Consider an open subgroup $O$ of $G$ such that $O \cap K$ has finite index in $K$; we claim that $O \cap E$ has finite index in $E$.  Let $L = O \cap H$.  We see that if $K$ is cocompact or of finite covolume in $H$, then so is its finite index subgroup $O \cap K$, and hence the overgroup $L$ of $O \cap K$ is respectively cocompact or of finite covolume in $H$.  Moreover, $L$ is open in $H$.  In a locally compact group, the only open subgroups that are cocompact or of finite covolume are the finite index ones, so $L$ has finite index in $H$.  But then since $E$ is a reduced envelope for $H$, it follows that $O \cap E$ has finite index in $E$ as claimed.  Thus $E$ is also a reduced envelope for $K$; in particular, $K \approx_G H$.
\end{proof}

We recall the following general feature of compactly generated \tdlc groups.

\begin{lem}[See for instance {\cite[Proposition~4.1]{CRW-Part2}}]\label{lem:compact_to_finite_gen}
Let $G$ be a compactly generated \tdlc group, let $U$ be a compact open subgroup of $G$ and let $D$ be a dense subgroup of $G$.  Then there exists a finite subset $F$ of $D$ such that $G = \langle F \rangle U$.
\end{lem}

We can now prove Theorems~\ref{intro:covol_envelopes} and \ref{intro:embedding_envelopes}.

\begin{proof}[Proof of Theorem~\ref{intro:covol_envelopes}]
Recall that $G$ is a \tdlc group, $A$ is a subgroup of $G$ such that $\overline{A}$ is cocompact or of finite covolume, and $H_1$ is a compactly generated subgroup of $G$.  Our aim is to find $H_2 \le A$ such that $H_1 \approx_G H_2$, and such that $H_2$ is finitely generated in the case that $\overline{A}$ is cocompact.

By Lemma~\ref{lem:envelope_equiv}(ii), $H_1 \approx_G E$ where $E$ is a reduced envelope for $H_1$.  So we may assume that $H_1$ is open in $G$.  Let $B = A \cap H_1$.  Since $H_1$ is open in $G$, we see that $\overline{B} = \overline{A} \cap H_1$.  Moreover, $\overline{B}$ is cocompact or of finite covolume in $H_1$.  Thus $B \approx_G H_1$ by Lemma~\ref{lem:envelope_equiv}(iii).

Let us now suppose $\overline{A}$ is cocompact in $G$.  In this case $\overline{B}$ is cocompact in $H_1$; in particular, $\overline{B}$ is compactly generated.  Then by Lemma~\ref{lem:compact_to_finite_gen}, there is a compact open subgroup $V$ of $\overline{B}$ and a finite subset $F$ of $B$ such that $\overline{B} = H_2V$, where $H_2 = \langle F \rangle$.  In particular, $H_2$ has cocompact closure in $\overline{B}$, so $H_2 \approx_G B$ and hence $H_2 \approx_G H_1$.
\end{proof}

\begin{proof}[Proof of Theorem~\ref{intro:embedding_envelopes}]
Recall that $G$ and $H$ are \tdlc groups and $\phi: G \rightarrow H$ is a continuous homomorphism, and $\mc{E}_G$ and $\mc{E}_H$ are the sets of commensurability classes of compactly generated open subgroups of $G$ and $H$ respectively.

Let $K_1$ and $K_2$ be compactly generated subgroups of $G$ such that $K_1 \approx_G K_2$.  Let $E_i$ be a reduced envelope in $G$ for $K_i$ for $i=1,2$.  Then both $\overline{\phi(K_1)}$ and $\overline{\phi(K_2)}$ are compactly generated, so each has a reduced envelope in $H$, and hence $\phi(K_1)$ and $\phi(K_2)$ have reduced envelopes.  Given an open subgroup $U_1$ of $H$ containing a finite index subgroup of $\phi(K_1)$, then $\phi\inv(U_1)$ is an open subgroup of $G$ containing a finite index subgroup of $K_1$.  Hence $\phi\inv(U_1)$ contains a finite index subgroup of $E_1$, and then by commensurability, $\phi\inv(U_1)$ contains a finite index subgroup of $E_2$, and in particular of $K_2$.  Thus $U_1$ contains a finite index subgroup of $\phi(K_2)$.  Similarly, any open subgroup of $H$ containing a finite index subgroup of $\phi(K_2)$ also contains a finite index subgroup of $\phi(K_1)$.  We conclude that the reduced envelopes for $\phi(K_1)$ and $\phi(K_2)$ are commensurate, so $\phi(K_1) \approx_H \phi(K_2)$.

Define $\theta: \mc{E}_G \rightarrow \mc{E}_H$ by $\theta(E_G(O)) = E_H(\phi(O))$.  To see that $\theta$ is well-defined, let $O_2$ be an element of $E_G(O)$ for some compactly generated open subgroup $O$ of $G$.  Then $O$ and $O_2$ are commensurate, so $O \approx_G O_2$ by Lemma~\ref{lem:envelope_equiv}(i).  Moreover, since $O$ is compactly generated, so is $O_2$.  Thus $E_H(\phi(O)) = E_H(\phi(O_2))$ by the previous paragraph.

Now suppose that $H = \phi(G)X$ for a compact set $X$ and let $U$ be a compactly generated open subgroup of $H$.  Then by Theorem~\ref{intro:covol_envelopes}, there is a finitely generated subgroup $V$ of $\phi(G)$ such that $V \approx_H U$.  Let $F$ be a finite generating set for $V$ and let $F'$ be a finite subset of $G$ such that $\phi(F') = F$.  Then $K = \langle F' \rangle$ is a finitely generated subgroup of $G$, so has a reduced envelope $E$ in $G$.  We have $K \approx_G E$, and hence $\phi(K) \approx_H \phi(E)$; moreover, $\phi(K) = V \approx_H U$, so in fact $\phi(E) \approx_H U$.  This demonstrates that $\theta$ is surjective.
\end{proof}

\subsection{Background on elementary groups}\label{sec:elementary}

For some context for later results in this article, we briefly recall the class $\ms{E}$ of elementary groups defined by P. Wesolek in \cite{Wesolek} and their canonical rank function, the decomposition rank.  See \cite{Wesolek} for more details.  From now on we write `\tdlcsc' to mean `\tdlc second-countable'.

\begin{defn}
The class of \defbold{elementary} \tdlcsc groups is the smallest class of \tdlcsc groups such that
	\begin{enumerate}[(i)]
		\item $\ms{E}$ contains all second countable profinite groups and countable discrete groups.
		\item $\ms{E}$ is closed under taking closed subgroups.
		\item $\ms{E}$ is closed under taking Hausdorff quotients.
		\item $\ms{E}$ is closed under forming extensions that result in a \tdlcsc group.
		\item If $G$ is a \tdlcsc group and $G=\bigcup_{i\in \bN}O_i$ where $(O_i)_{i\in \bN}$ is an $\subseteq$-increasing sequence of open subgroups of $G$ with $O_i\in\ms{E}$ for each $i$, then $G\in\ms{E}$.	\end{enumerate}
\end{defn} 

\begin{defn}\label{def:rank}
	Let $\ms{T}$ be the class of \tdlcsc groups and let $\omega_1$ be the first uncountable ordinal.  The \defbold{decomposition rank} $\xi$ is a partial ordinal-valued function from $\ms{T}$ to $[1,\omega_1)$ satisfying the following properties:
	\begin{enumerate}[(a)]
		\item $\xi(G)=1$ if and only if $G = \triv$;
		
		\item If $G \in \ms{T}$ is non-trivial and $G=\bigcup_{i\in \bN}O_i$ with $(O_i)_{i\in \bN}$ some increasing sequence of compactly generated open subgroups of $G$, then $\xi(G)$ is defined if and only if $\xi(R_i)$ is defined for $R_i := \Res(O_i)$ for all $i \in \bN$.  If $\xi(G)$ is defined, then
		\[
		\xi(G)=\sup_{i\in \bN}\xi(R_i)+1.
		\]
	\end{enumerate}
\end{defn}

By \cite[Theorem 4.7]{Wesolek} and \cite[Lemma 4.12]{Wesolek}, such a function $\xi$ exists, is unique, has domain of definition exactly $\ms{E}$, and is equivalent to the decomposition rank as defined in \cite{Wesolek}.

\subsection{Discrete regional quotient groups}\label{sec:drq}

We now give an example of how reduced envelopes can be used to reduce questions about arbitrary compactly generated closed subgroups $K$ of a \tdlc group $G$ to the case when $K$ is open.
 
\begin{prop}\label{prop:descent_groups}Let $G$ be a \tdlc group.  The following are equivalent:
\begin{enumerate}[(i)]
 \item For every nontrivial compactly generated closed subgroup $H$ of $G$, then $\Res_G(H) \ngeq H$;
 \item For every nontrivial compactly generated closed subgroup $H$ of $G$, then $\Res(H) \neq H$;
 \item For every noncompact compactly generated closed subgroup $H$ of $G$, then $H$ has an infinite discrete quotient;
 \item For every noncompact compactly generated open subgroup $H$ of $G$, then $H$ has an infinite discrete quotient.
\end{enumerate}
\end{prop}
 
\begin{proof}
Let us note that for any closed subgroup $H$ of $G$, then $O \cap H$ is an open normal subgroup of $H$ whenever $O$ is an open $H$-invariant subgroup of $G$.  In particular, $\Res(H) \le \Res_G(H)$, so (i) implies (ii).

We now claim that (ii) implies (iii), by showing the contrapositive.  Let $H$ be a noncompact compactly generated closed subgroup of $G$, let $R =\Res_G(H)$ and let $K = \overline{HR}$.  Suppose that $H$ has no infinite discrete quotient.  Then certainly $K/R$ has no infinite discrete quotient.  Let $E = KU$ be a compactly generated reduced envelope for $H$ in $G$, as given by Theorem~\ref{thm:sr_envelope}, where $U$ is a compact open subgroup of $G$.  Then $R = \Res(E)$, so $R$ is in the kernel of any map from $E$ to a discrete group.  Moreover, $K$ is cocompact in $E$, so given any continuous homomorphism $\phi: E \rightarrow D$ where $D$ is discrete, then $\phi(K)$ has finite index in $E$.  In particular, since $K/R$ has no infinite discrete quotient, we see that $E$ has no infinite discrete quotient, so $E/R$ is a \tdlc group with no infinite discrete quotient.  By Theorem~\ref{thm:sr_envelope}(\ref{sr_envelope:3}), $E/R$ is also a SIN group; the only way this is possible is if $E/R$ is compact.  Thus $R$ is cocompact in $E$, so $R$ is compactly generated.  Since $E$ contains $H$, $E$ is not compact and hence $R$ is nontrivial.  We now observe that $\Res(R) = \Res_R(E)$ by Lemma~\ref{lem:equi_residual}, whilst $\Res_R(E) = R$ by Proposition~\ref{prop:residual_distal}.  Thus $\Res(R) = R$, a contradiction to (ii).

It is clear that (iii) implies (iv).  We prove the remaining implication (iv) $\Rightarrow$ (i) in contrapositive form.  Suppose there is a nontrivial compactly generated closed subgroup $H$ of $G$ such that $\Res_G(H) \ge H$.  Certainly $H$ is not compact.  As in the previous paragraph, we can take a compactly generated reduced envelope $E = H\Res_G(H)U$ for $H$ in $G$.  This time, in fact $E = \Res_G(H)U$ and $\Res(E) = \Res_G(H)$.  We see that $E/\Res(E)$ is compact, ensuring that $E$ has no infinite discrete quotient, but $E$ is not compact, a contradiction to (iv). 
\end{proof}

Say a \tdlc group $G$ is a \defbold{discrete regional quotient group} if it satisfies any of the four equivalent conditions of Proposition~\ref{prop:descent_groups}.  Let $\ms{D}$ be the class of discrete regional quotient groups.  We also recall the class $\ms{S}$ of nondiscrete compactly generated topologically simple \tdlc groups.  Clearly the classes $\ms{D}$ and $\ms{S}$ are disjoint.  We see that $\ms{D}$ contains the class $\ms{E}$ of elementary groups; indeed this containment is strict (see Remark~\ref{rem:U(F)} below).  On the other hand, every second-countable group in $\ms{D}$ is isomorphic to a group in the set $E^*$ introduced in \cite[Theorem~3.12]{ReidFlat}.  It is not clear if there are groups in $E^*$ that are not in $\ms{D}$; one possibility is that there exists $G \in \ms{S}$ such that every element of $G$ has trivial contraction group.

By appealing to results of P.-E. Caprace and N. Monod (\cite{CM}), we obtain the following statement on the role of $\ms{D}$ and $\ms{S}$ in the structure of general \tdlc groups.

\begin{cor}\label{cor:DandS}
Let $G$ be a \tdlc group.  Then exactly one of the following holds:
\begin{enumerate}[(i)]
\item $G \in \ms{D}$, that is, every noncompact compactly generated closed subgroup of $G$ has an infinite discrete quotient;
\item There is a compactly generated open subgroup $K$ of $G$ such that $\Res(K)$ is cocompact in $K$, and such that there is a $K$-invariant closed subgroup $L$ of $\Res(K)$ with $\Res(K)/L \in \ms{S}$.
\end{enumerate}
\end{cor}

\begin{proof}
If (i) holds, then the only compactly generated open subgroups $K$ of $G$ such that $\Res(K)$ is cocompact in $K$ are the compact open subgroups of $G$; in this case $\Res(K) = \triv$, so $\Res(K)$ has no quotient in $\ms{S}$.  Thus (i) and (ii) are mutually exclusive.  We may therefore suppose (i) fails, that is, not every noncompact compactly generated closed subgroup of $G$ has an infinite discrete quotient.  Then by Proposition~\ref{prop:descent_groups}, there is a compactly generated open subgroup $K$ of $G$ such that $K$ has no infinite discrete quotient.  It then follows by \cite[Theorem~F]{CM} and \cite[Proposition~5.4]{CM} that $\Res(K)$ is cocompact in $K$ and that $\Res(K)$ has $0 < n < \infty$ quotients in $\ms{S}$.  By replacing $K$ with a finite index open subgroup of $K$, we can ensure that each of the quotients of $\Res(K)$ in $\ms{S}$ has $K$-invariant kernel.  (Note that if $K_2$ is a finite index open subgroup of $K$, then $\Res(K) = \Res(K_2)$.)
\end{proof}

Instead of considering all infinite discrete quotients, we can consider infinite quotients in a commensurability class of discrete groups.  Here again, the situation for compactly generated closed subgroups in general reduces to considering compactly generated open subgroups.

 \begin{prop}\label{prop:descent_groups:class}Let $G$ be a \tdlc group.  Let $\mc{C}$ be a class of infinite discrete groups, such that if $H \in \mc{C}$ and $K$ is a group isomorphic to a finite index subgroup of $H$, then $K \in \mc{C}$.  The following are equivalent:
 \begin{enumerate}[(i)]
 \item For every noncompact compactly generated closed subgroup $H$ of $G$, then $H$ has a quotient in $\mc{C}$;
 \item For every noncompact compactly generated open subgroup $H$ of $G$, then $H$ has a quotient in $\mc{C}$.
 \end{enumerate}
 \end{prop}
 
\begin{proof}
Clearly (i) implies (ii).  For the converse, suppose that (i) does not hold, that is, there is a noncompact compactly generated closed subgroup $H$ of $G$ with no quotient in $\mc{C}$.  Let $R =\Res_G(H)$ and let $K = \overline{HR}$.  Then certainly $K/R$ has no quotient in $\mc{C}$.  Let $E = KU$ be a compactly generated reduced envelope for $H$ in $G$, as given by Theorem~\ref{thm:sr_envelope}, where $U$ is a compact open subgroup of $G$.  Then $R = \Res(E)$, so $R$ is in the kernel of any map from $E$ to a discrete group.  Moreover, $K$ is cocompact in $E$, so given any continuous homomorphism $\phi: E \rightarrow D$ where $D$ is discrete, then $\phi(K)$ has finite index in $E$.  In particular, since $K/R$ has no quotient in $\mc{C}$, we see that $E$ has no quotient in $\mc{C}$.  This contradicts (ii) and completes the proof.
\end{proof}

\begin{rem}\label{rem:U(F)}
This subsection (\S\ref{sec:drq}) is inspired by the recent preprint \cite{CW}; I thank Pierre-Emmanuel Caprace and Phillip Wesolek for sharing with me preliminary versions of their results.  Caprace and Wesolek show the following: Given a Burger--Mozes universal group $U(F)$ with prescribed local action on a locally finite tree (see \cite{BurgerMozes}), such that the local action $F$ is not free, then $U(F)$ is nonelementary and indeed involves groups in $\ms{S}$ as subquotients (\cite[Corollary~4.17]{CW}).  However, if $F$ is nilpotent, then given a compactly generated closed subgroup $H$ of $U(F)$ that is not compact, then $H$ has an infinite discrete quotient (see \cite[Corollary~1.2]{CW}).  These examples $U(F)$, where $F$ is nilpotent and does not act freely, have the following consequences for the present discussion:
\begin{enumerate}[(1)]
\item The group $U(F)$ is in $\ms{D} \smallsetminus \ms{E}$, so $\ms{E}$ is a proper subclass of $\ms{D}$.  In particular, this shows that not every group in the set $E^*$ is elementary;  as noted in \cite{CW}, this provides a negative answer to \cite[Question~2]{ReidFlat}.
\item There are closed subgroups $K \unlhd H \le U(F)$ such that $H$ is compactly generated and $H/K \in \ms{S}$.  Note that in this case we still have $H \in \ms{D}$.  Thus the class $\ms{D}$ is not closed under quotients, and the condition that $\Res(K)$ be cocompact in $K$ is essential for the dichotomy of Corollary~\ref{cor:DandS}.
\end{enumerate}
\end{rem}

\section{Some special classes of closed subgroups}

\subsection{Introduction}

In this section we consider the problem of describing closed subgroups of a \tdlc group in terms of open subgroups.  We focus on two elementary operations for extending a family of closed subgroups.

\begin{defn}\label{def:RIO}
Let $G$ be a topological group and let $\mc{C}(G)$ be a class of closed subgroups of $G$.  Define $\mc{IC}(G)$ (the class of \defbold{intersection-$\mc{C}$} subgroups) to be the class of subgroups that are intersections of groups in $\mc{C}(G)$.  Define $\mc{RC}(G)$ (the class of \defbold{regionally-$\mc{C}$} subgroups) as follows: we have $H \in \mc{RC}(G)$ if there is a compactly generated open subgroup $U$ of $H$ such that, for every compactly generated group $K$ such that $U \le K \le H$, then $K \in \mc{C}(G)$.

Let $\mc{O}(G)$ denote the class of open subgroups of $G$.  If $H \in \mc{IO}(G)$, say $H$ is an \defbold{IO subgroup} of $G$, and if $H \in \mc{RIO}(G)$, say $H$ is a \defbold{RIO subgroup}.
\end{defn}

It is clear that $\mc{RO}(G) = \mc{O}(G)$.  If $G$ is a \tdlc SIN group, then $\mc{IO}(G)$ is the class of all closed subgroups, by Corollary~\ref{cor:SIN_intersection}; however, for more general \tdlc groups it is easy to find examples of closed subgroups that are not IO subgroups, for example lattices in a nondiscrete simple \tdlc group.  Given an arbitrary subgroup $H$ of the \tdlc group $G$, it is clear that there is a unique smallest IO subgroup of $G$ that contains $H$, the \defbold{IO-closure} in $G$, namely the intersection of all open subgroups that contain $H$; if $H$ has (SR) on $G$, then by Theorem~\ref{thm:sr_envelope}(\ref{sr_envelope:4}), the IO-closure of $H$ is exactly $\overline{H\Res_G(H)}$.  Note also that if $K$ is a closed \defbold{locally normal} subgroup of $G$, that is, such that $\N_G(K)$ is open, then we have $K \in \mc{IO}(G)$ by considering the open subgroups of $\N_G(K)/K$; so if $H$ has (SR) on $G$, then $\Res_G(H) \in \mc{IO}(G)$.

\subsection{Codistal subgroups}

To understand the closed subgroup structure of a \tdlc group $G$, it is useful to consider distal coset spaces of $G$, where now $G$ itself acts by left translation.

\begin{defn}
Let $G$ be a topological group and let $H$ be a closed subgroup of $G$.  Say $H$ is \defbold{codistal} in $G$, and write $H \in \mc{D}(G)$, if the action of $G$ on $G/H$ by left translation is distal.
\end{defn}

We remark at this point that orbits of more general distal actions of $G$ can be converted into distal coset spaces, at the cost of passing to a possibly finer topology: if $G$ acts continuously and distally on a topological space $X$, then each point stabilizer of the action is codistal in $G$.

The following characterization of codistal subgroups is essentially due to H. Keynes \cite[see Theorem~2.3]{Keynes}.

\begin{lem}\label{lem:KUK}
Let $G$ be a topological group, let $H$ be a closed subgroup of $G$ and let $G$ act on $G/H$ by left translation.  Then the following are equivalent:
\begin{enumerate}[(i)]
\item $G$ acts distally on $G/H$;
\item $H$ is a distal point for the action of $H$ on $G/H$;
\item We have $H = \bigcap_{U \in \mc{U}}HUH$, where $\mc{U}$ is a base of identity neighbourhoods in $G$.
\end{enumerate}
\end{lem}

\begin{proof}
Clearly (i) implies (ii), and (ii) implies (iii) by Lemma~\ref{lem:distal_equiv}.  Suppose (iii) holds.  We note that $G$ acts transitively by homeomorphisms on $G/H$, so either all diagonal points $(zH,zH)$ are accumulation points of nondiagonal orbits or none of them are.  Thus to show $G$ acts distally on $G/H$, it suffices to consider whether there exists a proximal pair $(xH,yH)$ such that $xH \neq yH$ and $(H,H) \in \overline{\{(gxH,gyH) \mid g \in G\}}$.  Given such a pair, we see that for all $U \in \mc{U}$, there exists $g \in G$ such that $gx,gy \in UH$.  It then follows that $x\inv y \in HU\inv UH$.  Since $U \in \mc{U}$ was arbitrary, we have $x\inv y \in \bigcap_{U \in \mc{U}}HUH$.  By (iii), we therefore have $x\inv y \in H$, that is, $xH = yH$, a contradiction.  Thus $G$ acts distally on $G/H$, so (iii) implies (i).
\end{proof}

A special case are closed commensurated subgroups.  Recall that for a subgroup $H$ of a group $G$, the \defbold{commensurator} of $H$ in $G$ is the group
\[
\Comm_G(H) := \{g \in G \mid |H:H \cap gHg\inv||gHg\inv:H \cap gHg\inv| < \infty\};
\]
we say $H$ is \defbold{commensurated} in $G$ if $\Comm_G(H) = G$.  An easy observation is that $H$ is commensurated in $G$ if and only if $H$ has finite orbits on the coset space $G/H$.  For closed subgroups, this is clearly a special case of the situation where $H$ has distal action on $G/H$.

Given Lemma~\ref{lem:KUK}, we observe some more closure properties of the set of codistal subgroups of $G$.  In particular, every IO-subgroup is codistal.

\begin{prop}\label{codistal:closure_properties}
Let $G$ be a topological group.
\begin{enumerate}[(i)]
\item\label{codistal:closure:1} Every closed commensurated subgroup of $G$ is codistal.
\item\label{codistal:closure:2} Let $H \in \mc{D}(G)$ and $K \in \mc{D}(H)$.  Then $K \in \mc{D}(G)$.
\item\label{codistal:closure:3} Let $H,K \le G$ be closed subgroups such that $K \in \mc{D}(G)$.  Then $H \cap K$ is codistal in $H$.
\item\label{codistal:closure:4} Let $\mc{H} \subseteq \mc{D}(G)$.  Then $\bigcap_{K \in \mc{H}}K \in \mc{D}(G)$.
\item\label{codistal:closure:5} Let $K \le H \le G$ be closed subgroups such that $K$ is cocompact in $H$ and codistal in $G$.  Then $H$ is codistal in $G$.
\item\label{codistal:closure:6} Every IO-subgroup of $G$ is codistal.
\end{enumerate}
\end{prop}

\begin{proof}
All actions in this proof are by left translation.  By Lemma~\ref{lem:KUK}, to show a subgroup $B$ of $A$ is codistal, it suffices to show that $B$ is a distal point for the action of $B$ on $A/B$.

(\ref{codistal:closure:1})
If $K$ is closed and commensurated in $G$, then $G/K$ is Hausdorff and the action of $K$ on $G/K$ has finite orbits: indeed, we see that for each $g \in G$, the right coset $Kg$ of $K$ is a left coset of $g\inv Kg$, hence a union of finitely many left cosets of $K \cap g\inv Kg$, and so the double coset $KgK$ is a union of finitely many left cosets of $K$.  Thus the orbits of $K$ on $G/K$ have no accumulation points; in particular, $K$ is a distal point for the action of $K$ on $G/K$.

(\ref{codistal:closure:2})
Suppose $x \in G$ is such that the $K$-orbit of $xK$ accumulates at the trivial coset in $G/K$.  Then the $K$-orbit of $xH$ accumulates at the trivial coset in $G/H$.  Since $K$ acts distally on $G/H$, it follows that $x \in H$.  It then follows that in fact $x \in K$, since $K$ acts distally on $H/K$.  Thus $K$ is codistal in $G$.

(\ref{codistal:closure:3})
Let $L = H \cap K$.  There is then an $H$-equivariant continuous injective map from $H/L$ to $G/K$, where $H$ acts on both coset spaces by translation.  Since $H$ acts distally on $G/K$, it also acts distally on $H/L$.  Thus $L$ is codistal in $G$.

(\ref{codistal:closure:4})
Let $H = \bigcap_{K \in \mc{H}}K$ and suppose $x \in G$ is such that the $H$-orbit of $xH$ accumulates at the trivial coset in $G/H$.  Then for each $K \in \mc{H}$, the $H$-orbit of $xK$ accumulates at the trivial coset in $G/K$.  Since $H$ acts distally on $G/K$, it follows that $x \in K$.  Since $K \in \mc{H}$ was arbitrary, we must have $x \in H$.  Thus $H$ is codistal in $G$.

(\ref{codistal:closure:5})
Suppose $x \in G$ is such that the $H$-orbit of $xH$ accumulates at the trivial coset in $G/H$.  Then there is a net $(h_i)_{i \in I}$ in $H$ such that $h_ixH$ converges to $H$.  Let $U$ be a compact open subgroup of $G$.  Then the net $(h_ixK)_{i \in I}$ is eventually confined to a compact subset $UH/K$ of $G/K$; by passing to a subnet we may assume $h_ixK \rightarrow yK$ for some $y \in G$.  Indeed, since $H$ is closed and $h_ixH \rightarrow H$, we must have $y \in H$.  Since $K \in \mc{D}(G)$, we see that $H$ acts distally on $G/K$; the $H$-orbit of $yK$ is the compact set $H/K$.  We now apply Lemma~\ref{lem:distal_orbit_closure} to conclude that $\overline{HxK} = H$; in particular, $x \in H$.  Thus the trivial coset is a distal point for the action of $H$ on $G/H$, and hence $H$ is codistal in $G$.

(\ref{codistal:closure:6})
Open subgroups are clearly codistal, so (\ref{codistal:closure:6}) follows from (\ref{codistal:closure:4}).
\end{proof}

In general there are codistal subgroups of \tdlc groups, even codistal cocompact subgroups, that are not IO-subgroups; see Example~\ref{ex:different_classes}(ii).  However, if we focus on the case of subgroups with equicontinuously generated action, the possibilities are more restricted.  In particular, the next two lemmas will show that an IO-subgroup with equicontinuously generated action is very closely approximated by an equicontinuously generated open subgroup.

\begin{defn}
Given a topological group $G$ and $H \le \Aut(G)$, write $\Res_{G,f}(H)$ for the intersection of all open $H$-invariant subgroups of finite index in $G$.  In the case that $H = \Inn(G)$ we define $\Res_{G,f}(H) =: \Res_f(G)$.
\end{defn}

\begin{lem}\label{lem:cocompact_IO}
Let $G$ be a \tdlc group and let $H \in \mc{IO}(G)$.  Suppose that $H$ is cocompact in $G$.  Then $H$ is an intersection of open subgroups of finite index in $G$.  As a consequence, we have $\Res(G) \le \Res_f(G) \le H$.
\end{lem}

\begin{proof}
Let $H \le O \le G$ such that $O$ is open.  Then the coset space $G/O$ is both compact and discrete, hence it is finite.  Since $H \in \mc{IO}(G)$, it follows that $H$ is an intersection of open subgroups $O$ of finite index in $G$.  For each such $O$, the group $\bigcap_{g \in G}gOg\inv$ is itself open, since it is a finite intersection of conjugates of $O$.  Thus $\Res_f(G) \le H$.  In turn, it is clear that $\Res(G) \le \Res_f(G)$.
\end{proof}

\begin{lem}\label{lem:IO_equi_gen}
Let $G$ be a \tdlc group and let $H$ be a closed subgroup of $G$ that has equicontinuously generated action on $G$ (for example, any closed compactly generated subgroup of $G$).  Then the following are equivalent:
\begin{enumerate}[(i)]
\item $\Res_G(H) = \Res(H)$;
\item $\Res_G(H) \le H$;
\item $H \in \mc{IO}(G)$;
\item $H \in \mc{D}(G)$;
\item There is an open subgroup $E$ of $G$ such that $\Res(E) \le H \le E$.
\end{enumerate}
Moreover, if (i)--(v) hold, then in fact there is an equicontinuously generated open subgroup $E$ of $G$ such that $\Res(E) = \Res(H)$, $H$ is cocompact in $E$ and $H$ is an intersection of finite index open subgroups of $E$.\end{lem}

\begin{proof}
Let $R = \Res_G(H)$.  It is clear that (i) implies (ii).  By Theorem~\ref{thm:sr_envelope}(\ref{sr_envelope:4}), the IO-closure of $H$ in $G$ is exactly $\overline{HR}$, so (ii) and (iii) are equivalent.  We see by Proposition~\ref{codistal:closure_properties}(\ref{codistal:closure:6}) that (iii) implies (iv).

Suppose (iv) holds.  Then $H$ acts distally on $G/H$ by conjugation, so (ii) holds by Theorem~\ref{intro:distal_SIN}, or in other words, $H = \overline{HR}$.  It then follows by Theorem~\ref{thm:sr_envelope} that $H$ has a reduced envelope of the form $E = HU$ where $U$ is a compact open subgroup of $G$.  In particular, $H$ is cocompact in $E$.  Since (ii) implies (iii), we see that $H$ is a cocompact intersection of open subgroups of $E$; hence $\Res(E) \le H$ by Lemma~\ref{lem:cocompact_IO}, proving (v).

Suppose (v) holds.  We have $R = \Res_E(H)$, and clearly $\Res_E(H) \le \Res(E)$, so $R \le H$.  If $O$ is an open $H$-invariant subgroup of $G$, then $H \cap O$ is an open normal subgroup of $H$; thus $\Res(H) \le R$.  On the other hand, if $L$ is an open normal subgroup of $H$, then $L \cap R$ is an open $H$-invariant subgroup of $R$, so by Proposition~\ref{prop:residual_distal}, $L \ge R$; hence $R = \Res(H)$.  We have now shown that (v) implies (i), which completes the cycle of implications to show that (i)--(v) are equivalent.

Now suppose (i)--(v) holds.  By Theorem~\ref{thm:sr_envelope}, $H$ has a reduced envelope $E = \overline{HRU}$ where $U$ is a compact open subgroup; in fact $E = HU$ in this case, since $R \le H$, so $H$ is cocompact in $E$.  By Theorem~\ref{thm:sr_envelope}(\ref{sr_envelope:1}), $E$ is equicontinuously generated on $G$, and hence on itself.  We have $R = \Res(E)$ by Theorem~\ref{thm:sr_envelope}(\ref{sr_envelope:2}), so $R = \Res(H)$ by condition (ii) of the present lemma.  By (iii), $H$ is an intersection of open subgroups of $E$; since $H$ is cocompact in $E$, it follows by Lemma~\ref{lem:cocompact_IO} that $H$ is an intersection of finite index open subgroups of $E$.  We have thus proved that $E$ has all the required properties.
\end{proof}

It was noted in \cite[Lemma~7.4]{CRW-Part1} that if $G$ is a first-countable profinite group and $H$ is a commensurated compact subgroup of $G$, then $G$ normalizes a finite index subgroup of $H$.  We now prove a similar result in the more general context of compactly generated closed subgroups of \tdlc groups.

\begin{prop}\label{prop:commensurated_locnorm}
Let $G$ be a \tdlc group and let $H$ be a compactly generated closed subgroup of $G$.  Suppose that $H$ has at most countably many open subgroups of finite index (for example, $H$ is second-countable).  Then the following are equivalent:
\begin{enumerate}[(i)]
\item $\Comm_G(H)$ is open in $G$;
\item there is an open normal subgroup of $H$ of finite index such that $\N_G(K)$ is open in $G$.
\end{enumerate}
\end{prop}

\begin{proof}
Suppose $\Comm_G(H)$ is open in $G$.  By Proposition~\ref{codistal:closure_properties}(\ref{codistal:closure:1}), we see that $H$ is codistal in $\Comm_G(H)$; hence by Lemma~\ref{lem:IO_equi_gen}, there is an open subgroup $E$ of $\Comm_G(H)$ such that $H$ is cocompact in $E$ and $H$ is an intersection of finite index open subgroups of $E$.

By \cite[Main Theorem]{CKRW}, it follows that given a finite subset $X$ of $E$, there is a finite index subgroup $K_X$ of $H$ that is normal in $\langle H,X \rangle$.  We can clearly replace $K_X$ with its closure in $H$, and so ensure that $K_X$ is closed, hence open in $H$.  Thus $E = \bigcup_{K \in \mc{K}}\N_E(K)$ where $\mc{K}$ is the set of open normal subgroups of $H$ of finite index.  Each of the subgroups $\N_E(K)$ is closed, and by hypothesis, $\mc{K}$ is countable.  Hence by the Baire Category Theorem, there is some $K \in \mc{K}$ such that $\N_E(K)$ has nonempty interior in $E$, and hence is an open subgroup of $E$, which in turn means that $\N_G(K)$ is open in $G$.  Thus (i) implies (ii).

Conversely, suppose $K$ is an open normal subgroup of $H$ of finite index such that $\N_G(K)$ is open in $G$.  Then $\Comm_G(K) \ge \N_G(K)$, so $\Comm_G(K)$ is open, and in fact $\Comm_G(K) = \Comm_G(H)$, since $K$ is commensurate with $H$.  Thus $\Comm_G(H)$ is open, showing that (ii) implies (i).
\end{proof}

\subsection{RIO subgroups}\label{sec:RIO}

We now consider the class $\mc{RIO}(G)$ of RIO-subgroups of a \tdlc group $G$.  This class has a number of closure properties that can be derived from Theorem~\ref{intro:distal_SIN} and Proposition~\ref{codistal:closure_properties}, as will be shown in Theorem~\ref{thm:rio_closure_properties} below.

Here are some characterizations of RIO subgroups.  Say that a net of subgroups $(H_i)_{i \in I}$ is \defbold{ascending} if $i \le j$ implies $H_i \le H_j$, and \defbold{descending} if $i \le j$ implies $H_i \ge H_j$.

\begin{prop}\label{prop:RIO_equivalents}
Let $G$ be a \tdlc group and let $H$ be a closed subgroup of $G$.  Then the following are equivalent:
\begin{enumerate}[(i)]
\item $H \in \mc{RIO}(G)$;
\item $H \in \mc{RD}(G)$;
\item  there is an ascending net $(H_i)_{i \in I}$ of subgroups such that $H = \bigcup_{i \in I}H_i$ and $H_i \in \mc{D}(G) \cap \mc{O}(H)$ for all $i \in I$;
\item every compactly generated open subgroup of $H$ is in $\mc{IO}(G)$;
\item every compactly generated open subgroup of $H$ is in $\mc{D}(G)$.
\end{enumerate}
\end{prop}

\begin{proof}
By Lemma~\ref{lem:IO_equi_gen}, we see that (i) and (ii) are equivalent and (iv) and (v) are equivalent.  The implications (ii) $\Rightarrow$ (iii) and (iv) $\Rightarrow$ (i) are clear.  It now suffices to show (iii) $\Rightarrow$ (v).

Suppose (iii) holds for the ascending net $(H_i)_{i \in I}$ and let $K$ be a compactly generated open subgroup of $H$.  Then for $i \in I$ sufficiently large, we have $K \le H_i$.  Now $K$ is open in $H_i$, so certainly $K \in \mc{D}(H_i)$, whilst $H_i \in \mc{D}(G)$.  Thus $K \in \mc{D}(G)$ by Proposition~\ref{codistal:closure_properties}(\ref{codistal:closure:2}).  Thus (iii) implies (v) as required.
\end{proof}

As a generalization of directed unions and intersections of subgroups, one can define the limit inferior of a net of subgroups. 

\begin{defn}\label{def:liminf}
Let $G$ be a group and let $(H_i)_{i \in I}$ be a net of subgroups of $G$.  Define $\liminf (H_i)_{i \in I}$ to be the set of elements $g$ of $G$ such that there exists $i \in I$ for which $g \in H_j$ for all $j \ge i$.
\end{defn}

Note that the limit inferior of a net of closed subgroups is itself a subgroup, but not necessarily closed.  (The limit superior is less useful in this context, as it is not even a subgroup in general.)  In the special case that the net is ascending or descending, the limit inferior is just the union or intersection respectively of the net.

We now establish several closure properties of the class $\mc{RIO}(G)$, including properties analogous to those established for the smaller class $\mc{D}(G)$.  This will in particular prove Theorem~\ref{thmintro:rio_closure_properties}.

\begin{thm}\label{thm:rio_closure_properties}
Let $G$ be a \tdlc group.
\begin{enumerate}[(i)]
\item\label{rio_closure_properties:1} Every codistal subgroup of $G$ is a RIO subgroup.
\item\label{rio_closure_properties:2} Let $K \le H \le G$ be closed subgroups.  If $K \in \mc{RIO}(H)$ and $H \in \mc{RIO}(G)$, then $K \in \mc{RIO}(G)$.
\item\label{rio_closure_properties:3} Let $H,K \le G$ be closed subgroups.  If $K \in \mc{RIO}(G)$, then $H \cap K \in \mc{RIO}(H)$.
\item\label{rio_closure_properties:4} Let $\mc{H} \subseteq \mc{RIO}(G)$ and let $H = \bigcap_{K \in \mc{H}}K$.  Then $H \in \mc{RIO}(G)$.
\item\label{rio_closure_properties:5} Let $K \le H \le G$ be closed subgroups such that $K$ is cocompact in $H$.  If $K \in \mc{RIO}(G)$, then $H \in \mc{RIO}(G)$.
\item\label{rio_closure_properties:6} Let $(H_i)_{i \in I}$ be a net of RIO subgroups of $G$.  Then $H = \overline{\liminf (H_i)_{i \in I}}$ is a RIO subgroup of $G$.
\item\label{rio_closure_properties:7} Let $\psi: H \rightarrow G$ be a continuous homomorphism, where $H$ is a \tdlc group, and let $K$ be a RIO subgroup of $G$.  Then $\psi\inv(K)$ is a RIO subgroup of $H$.
\end{enumerate}
\end{thm}

\begin{proof}
In this proof we will use repeatedly without further comment the following special case of Proposition~\ref{prop:RIO_equivalents}: given $H \le G$ closed, then $H \in \mc{RIO}(G)$ if and only if every compactly generated open subgroup of $H$ is in $\mc{D}(G)$.

(\ref{rio_closure_properties:1})
Let $H$ be a codistal subgroup of $G$ and let $L$ be a compactly generated open subgroup of $H$.  Then $L \in \mc{D}(H)$, so $L \in \mc{D}(G)$ by Proposition~\ref{codistal:closure_properties}(\ref{codistal:closure:2}).  Thus $H \in \mc{RIO}(G)$.

(\ref{rio_closure_properties:2})
Let $L$ be a compactly generated open subgroup of $K$ and let $M$ be a compactly generated open subgroup of $H$ containing $L$.  Then $L \in \mc{D}(H)$, so $L \in \mc{D}(M)$.  In turn, since $H \in \mc{RIO}(G)$, we have $M \in \mc{D}(G)$.  Thus $L \in \mc{D}(G)$ by Proposition~\ref{codistal:closure_properties}(\ref{codistal:closure:2}).  We conclude that $K \in \mc{RIO}(G)$ as required.

(\ref{rio_closure_properties:3})
Let $L$ be a compactly generated open subgroup of $H \cap K$ and let $M$ be a compactly generated open subgroup of $K$ containing $L$.  Then $M \in \mc{D}(G)$; hence $H \cap M \in \mc{D}(H)$ by Proposition~\ref{codistal:closure_properties}(\ref{codistal:closure:3}).  Since $L$ is open in $H \cap M$, it follows by Proposition~\ref{codistal:closure_properties}(\ref{codistal:closure:2}) that $L \in \mc{D}(H)$.  We conclude that $H \cap K \in \mc{RIO}(H)$.

(\ref{rio_closure_properties:4})
Let $L$ be a compactly generated open subgroup of $H$, let $K \in \mc{H}$ and let $W$ be a compact open subgroup of $K$.  The group $L' = \langle L,W \rangle$ is then a compactly generated open subgroup of $K$, so $L' \in \mc{IO}(G)$ by part (\ref{rio_closure_properties:2}).  Consequently, $\Res_G(L') \le L' \le K$ and in particular $\Res_G(L) \le K$.  Since $K \in \mc{H}$ was arbitrary, in fact $\Res_G(L) \le H$.  Now $L$ has (SR) on $G$ by Corollary~\ref{cor:equi_gen_res}, so $L'' := \overline{L\Res_G(L)} \in \mc{IO}(G)$ by Theorem~\ref{thm:sr_envelope}(\ref{sr_envelope:4}).  We conclude that $H$ is an ascending open union of codistal subgroups of $G$, so $H \in \mc{RIO}(G)$ by Proposition~\ref{prop:RIO_equivalents}.

(\ref{rio_closure_properties:5})
Let $L$ be a compactly generated open subgroup of $H$.  Then $L \cap K$ is open in $K$ and cocompact in $L$; in particular, $L \cap K$ is compactly generated.  Since $K \in \mc{RIO}(G)$ we therefore have $L \cap K \in \mc{D}(G)$; we then have $L \in \mc{D}(G)$ by Proposition~\ref{codistal:closure_properties}(\ref{codistal:closure:5}).  Thus $H \in \mc{RIO}(G)$.

(\ref{rio_closure_properties:6})
Let $D = \liminf (H_i)_{i \in I}$.  Then we can write $D$ as an ascending union of intersections of subgroups, namely $D = \bigcup_{i \in I}H_{\ge i}$, where $H_{ \ge i} = \bigcap_{j \ge i}H_j$.  By part (\ref{rio_closure_properties:4}), we have $H_{\ge i} \in \mc{RIO}(G)$ for all $i \in I$.

Let $L$ be a compactly generated open subgroup of $H$.  Since $D$ is dense in $H$, we see that $D \cap L$ is dense in $L$, and consequently there is a finite subset $S$ of $D$ and a compact open subgroup $W$ of $L$ such that $L = \langle S \rangle W$.  Since $S$ is finite we have $S \subseteq H_{\ge i}$ for some $i \in I$.  Let $M = L \cap H_{\ge i}$; note that $M$ is open in $H_{\ge i}$, since $L$ is open in $H$, so $M \in \mc{RIO}(G)$ by part (\ref{rio_closure_properties:2}).  Moreover, since $\langle S \rangle \le M$ and $L = \langle S \rangle W$, we see that $M$ is cocompact in $L$ and thus compactly generated; hence in fact $M \in \mc{D}(G)$.  Since $M$ is cocompact in $L$, it follows by Proposition~\ref{codistal:closure_properties}(\ref{codistal:closure:5}) that $L \in \mc{D}(G)$.  Thus $H \in \mc{RIO}(G)$.

(\ref{rio_closure_properties:7})
We can write $K = \bigcup_{i \in I}K_i$ where each $K_i$ is an intersection of open subgroups $(O_{ij})$ of $G$.  Then $\psi\inv(K) = \bigcup_{i \in I}\psi\inv(K_i)$ and $\psi\inv(K_i) = \bigcap_j \psi\inv(O_{ij})$; by continuity, each of the groups $\psi\inv(O_{ij})$ is open in $H$, so each of the groups $\psi\inv(K_i)$ is an IO subgroup of $H$.  Thus $\psi\inv(K)$ is a RIO subgroup of $H$ by Proposition~\ref{prop:RIO_equivalents}.
\end{proof}

As a corollary of Theorem~\ref{thm:rio_closure_properties}(\ref{rio_closure_properties:6}), we have yet another characterization of RIO subgroups.

\begin{cor}\label{cor:RIO_liminf}
Let $G$ be a \tdlc group and let $H$ be a closed subgroup of $G$.  Then $H \in \mc{RIO}(G)$ if and only if $H$ is the limit inferior of a net of open subgroups of $G$.
\end{cor}

\begin{proof}
If $H$ is the limit inferior of a net of open subgroups of $G$, then $H \in \mc{RIO}(G)$ by Theorem~\ref{thm:rio_closure_properties}(\ref{rio_closure_properties:6}).  Conversely, if $H \in \mc{RIO}(G)$, then $H$ is the union of an ascending net $(K_i)_{i \in I}$ of subgroups such that $K_i \in \mc{IO}(G)$.  In turn, each $K_i$ is the intersection of a collection $\mc{O}_i$ of open subgroups of $G$; we may ensure $\mc{O}_i$ is closed under taking finite intersections, and thus organize $\mc{O}_i$ as a descending net $(O_j)_{j \in J_i}$, so that in fact $K_i = \liminf (O_j)_{j \in J_i}$.  Now equip $J:= \bigsqcup_{i \in I}J_i$ with the ordering so that $j < j'$ exactly if either $j \in J_i$ and $j' \in J_{i'}$ for $i <_I i'$, or $j,j' \in J_i$ with $j <_{J_i} j'$.  Then $H = \liminf (O_j)_{j \in J}$.
\end{proof}

We finish this subsection by considering how the class $\mc{RIO}(G)$ relates to the class of all closed subgroups.  At this point it is instructive to consider the class of \tdlc groups of rank $2$. 

\begin{defn}Let $G$ be a nontrivial \tdlc group.  Then $G$ is \defbold{regionally SIN}, or has \defbold{rank $2$}, if every compactly generated subgroup of $G$ is a SIN group.\end{defn}

The term `rank $2$' is motivated by elementary \tdlcsc groups (recall \S\ref{sec:elementary}): if $G$ is a nontrivial second-countable \tdlc group, then $G$ is regionally SIN if and only if $G$ is elementary and $\xi(G) = 2$.

If $G$ is regionally SIN, then every compactly generated subgroup is an IO subgroup, so every closed subgroup is a RIO subgroup.  On the other hand, if $G$ is a \tdlc group such that every closed subgroup is a RIO subgroup, then certainly every compactly generated \emph{discrete} subgroup has SIN action (equivalently, equicontinuous action) on $G$.  The author is not aware of examples of \tdlc groups that are not regionally SIN, but such that every finitely generated discrete subgroup acts equicontinuously.

We can give a constructive description of the RIO-closure of an arbitrary subgroup, generalizing the IO-closure of a compactly generated subgroup.  Given a \tdlc group $G$, the \defbold{rank $2$ residual} $\Res_2(G)$ is given by $\Res_2(G) := \ol{\bigcup_{K \in \mc{K}}\Res(K)}$, where $\mc{K}$ is the set of compactly generated open subgroups of $G$; equivalently, $G/\Res_2(G)$ is the largest rank $2$ quotient of $G$.  Similarly the \defbold{relative rank $2$ residual} $\Res_{G,2}(K)$ of a closed subgroup $K$ of $G$ is $\ol{\bigcup_{L \in \mc{K}}\Res_G(L)}$, where now $\mc{K}$ is the set of compactly generated open subgroups of $K$.

\begin{prop}\label{prop:rio_res2}
$K$ is a RIO subgroup of $G$ if and only if $\Res_{G,2}(K) \le K$.  Given any subgroup $K$ of $G$, then $L = \ol{\Res_{G,2}(K)K}$ is the smallest RIO subgroup of $G$ that contains $K$.
\end{prop}

\begin{proof}
If $K$ is a RIO subgroup of $G$, then every compactly generated open subgroup $M$ of $G$ satisfies $\Res_G(M) \le M$, so certainly $\Res_{G,2}(K) \le K$.

Now let $K$ be any closed subgroup of $G$.  Given $L \ge K$, every compactly generated open subgroup of $K$ is contained in such a subgroup of $L$, so that $\Res_{G,2}(L) \ge \Res_{G,2}(K)$.  Thus any RIO subgroup of $G$ containing $K$ must contain $\ol{\Res_{G,2}(K)K}$.  It remains to show that $L = \ol{\Res_{G,2}(K)K}$ is a RIO subgroup of $G$.  We observe that $L$ is generated topologically by subgroups of the form $M' = \ol{M\Res_G(M)}$, where $M$ is a compactly generated open subgroup of $K$; moreover, such subgroups $M'$ form an ascending net.  Each group of this form is an IO subgroup of $G$ by Theorem~\ref{thm:sr_envelope}(\ref{sr_envelope:4}).  Thus by Theorem~\ref{thm:rio_closure_properties}(\ref{rio_closure_properties:6}), $L$ is a RIO subgroup of $G$ as required.
\end{proof}

\subsection{Regional relative separability}\label{sec:reg_sep}

For some closed subgroups $H$ of a \tdlc group $G$, there is a clear relationship between compactly generated open subgroups of $H$ and those of $G$, defined as follows.

\begin{defn}
Let $G$ be a \tdlc group and let $H$ be a closed subgroup of $G$.  Then $H$ is \defbold{regionally relatively separable (RRS)} in $G$ if for every compactly generated open subgroup $K$ of $H$, there is an open subgroup $O$ of $G$ such that $K = H \cap O$.  Write $\mathrm{RRS}(G)$ for the class of RRS subgroups of $G$.
\end{defn}

Note that the open subgroup $O$ in the above definition can always be taken to be compactly generated: one can replace $O$ with $O' = \langle K,U\rangle$ where $U$ is a compact open subgroup of $O$.

We can now show that all RIO subgroups are RRS.

\begin{proof}[Proof of Theorem~\ref{thmintro:RIO_separation}]
Let $H \in \mc{RIO}(G)$ and let $K$ be a compactly generated open subgroup of $H$.  By Proposition~\ref{prop:RIO_equivalents}, $K \in \mc{IO}(G)$; in particular, $\Res_G(K) \le K$.  It follows by Theorem~\ref{thm:sr_envelope}(\ref{sr_envelope:4}) that $\{W/K \mid W \in \mc{W}\}$ forms a base of neighbourhoods of the trivial coset in $G/K$, where $\mc{W}$ is the set of open subgroups containing $K$.  Since $K$ is open in $H$, the subspace $H/K$ of $G/K$ is discrete, so there is $O \in \mc{W}$ such that $H/K \cap O/K = K/K$, that is, $H \cap O = K$.
\end{proof}

As corollaries, we note that the locally normal property is inherited by compactly generated open subgroups, and also obtain another property of compactly generated IO-subgroups.

\begin{cor}\label{locnorm:trans}
Let $G$ be a \tdlc group, let $H$ be a closed locally normal subgroup of $G$, and let $K$ be a compactly generated open subgroup of $H$.  Then $K$ is locally normal in $G$.
\end{cor}

\begin{proof}
Without loss of generality we can replace $G$ with any open subgroup that contains $H$; thus we may assume $H$ is normal in $G$.  We then have $H \in \mc{RIO}(G)$ (indeed, $H \in \mc{IO}(G)$), so $H \in \mathrm{RRS}(G)$ by Theorem~\ref{thmintro:RIO_separation}.  In particular, there is an open subgroup $O$ of $G$ such that $K = H \cap O$.  Now $O$ clearly normalizes its intersection with the normal subgroup $H$, so we see that $K$ is normal in $O$.  In particular, $K$ is locally normal in $G$.
\end{proof}

\begin{cor}\label{IO_finite_residual}
Let $G$ be a \tdlc group and let $H$ be a closed compactly generated subgroup of $G$ such that $H \in \mc{IO}(G)$.  Let $E$ be a reduced envelope for $H$.  Then $\Res_f(E) = \Res_f(H)$; in particular, $\Res_f(H)$ is locally normal in $G$.
\end{cor}

\begin{proof}
By Theorem~\ref{thm:sr_envelope}, $H$ is cocompact in $E$, so by Lemma~\ref{lem:cocompact_IO} we have $\Res_f(E) \le H$; it is clear that $\Res_f(H) \le \Res_f(E)$.  Now let $K$ be an open subgroup of $H$ of finite index.  Then $K$ is compactly generated, so by Theorem~\ref{thmintro:RIO_separation} there is an open subgroup $O$ of $E$ such that $K = H \cap O$.  We see that $K$ is cocompact in $E$, so $O$ is also cocompact and hence of finite index.  The group $O^* = \bigcap_{g \in E}gOg\inv$ is then an open normal subgroup of $E$ of finite index such that $O^* \le O$, so $\Res_f(E) \le O$.  Thus $\Res_f(E) \le H \cap O = K$.  Since $K$ was an arbitrary open subgroup of $H$ of finite index, it follows that $\Res_f(E) \le \Res_f(H)$ and hence $\Res_f(E) = \Res_f(H)$.
\end{proof}

It is not clear which of the closure properties of the class of RIO subgroups also apply to RRS subgroups.  For instance, if $H$ is an ascending union of open subgroups, each RRS in $G$, is $H$ necessarily RRS in $G$?  We do at least observe that RRS is a transitive property.

\begin{prop}
Let $G$ be a \tdlc group and let $K \le H$ be closed subgroups of $G$.
\begin{enumerate}[(i)]
\item If $K$ is RRS in $G$, then $K$ is also RRS in $H$.
\item If $K$ is RRS in $H$ and $H$ is RRS in $G$, then $K$ is RRS in $G$.
\end{enumerate}
\end{prop}

\begin{proof}
(i)
Let $L$ be a compactly generated open subgroup of $K$.  By the RRS property, $L = K \cap O$ where $O$ is an open subgroup of $G$.  Since $K \le H$, in fact $L = K \cap (O \cap H)$, where now $O \cap H$ is an open subgroup of $H$.

(ii)
Let $L$ be a compactly generated open subgroup of $K$.  Then $L = K \cap M$ where $M$ is an open subgroup of $H$; as observed, we can take $M$ to be compactly generated.  In turn $M = H \cap O$ where $O$ is an open subgroup of $G$.  Thus $L = K \cap (H \cap O) = K \cap O$, showing that $K$ is RRS in $G$.
\end{proof}

To summarize, we have the following nested sequence of classes of closed subgroups of a \tdlc group $G$, where $\mathrm{LN}(G)$ is the class of closed locally normal subgroups and $\mc{C}(G)$ is the class of all closed subgroups:
\begin{equation}\label{eq:classes}
\mc{O}(G) \subseteq \mathrm{LN}(G) \subseteq \mc{IO}(G) \subseteq \mc{D}(G) \subseteq \mc{RIO}(G) \subseteq \mathrm{RRS}(G) \subseteq \mc{C}(G).
\end{equation}

Any of the inclusions in (\ref{eq:classes}) can be proper.  The class of closed subnormal subgroups is contained in $\mc{D}(G)$ by Proposition~\ref{codistal:closure_properties}(\ref{codistal:closure:2}), but is incomparable in general with the classes below $\mc{D}(G)$ in (\ref{eq:classes}).  We justify these claims with the examples below.

\begin{ex}\label{ex:different_classes} \

\begin{enumerate}[(i)]
\item
Let $A = \mathrm{Sym}(3)$, let $S$ be the symmetric group acting on $\bN$, and form the unrestricted wreath product $G = A \wr_{\bN} S$.  Here we embed $P = \prod_{\bN}A$ as a compact open subgroup and equip $S$ with the discrete topology.
\begin{enumerate}[(a)]
\item The trivial subgroup is closed and locally normal, but not open in $G$.
\item The subgroup $\{f : \bN \rightarrow A \mid  \forall i: f(i) \in \langle(12)\rangle\}$ of $P$ is an IO subgroup of $G$ that is not locally normal.
\item The subgroup $\{f \in S \mid \exists j \in \bN \; \forall i > j: f(i) = i\}$ of $S$ is a closed RIO subgroup of $G$ that is not codistal.
\item The subgroup $H = \{f \in S \mid \exists n \; \forall i: f(i) = i+n\}$ of $S$ is a closed RRS subgroup of $G$ that is not a RIO subgroup.
\item The subgroup $P \rtimes H$ of $G$ is an open subgroup that is not subnormal.
\end{enumerate}
\item
Let $V \cong W \cong \bF_p[[t]]$, regarded as profinite vector spaces over $\bF_p$.  Let $H$ be the group of continuous $\bF_p$-linear maps from $V$ to $W$ under pointwise addition, and define an action $\rho$ of $H$ on $K = V \oplus W$ by setting $\rho(h)(v,w) = (v,w+h(v))$.  Let $G = K \rtimes_{\rho} H$ where $K$ is embedded as a compact open subgroup of $G$ and $H$ carries the discrete topology.  Then $H$ is subnormal in $G$ (since $G$ is nilpotent), and in addition $H$ is cocompact in $G$, but $H$ is not an IO subgroup, since every open subgroup containing $H$ also contains $W$.
\item
Let $T$ be the regular locally finite tree of degree $3$, let $G$ be the group of type-preserving automorphisms of $T$, let $g$ be a hyperbolic element of $G$ and let $H = \langle g \rangle$.  Then $H$ is a closed (indeed discrete) subgroup of $G$, but not an RRS subgroup: indeed, $H$ has trivial intersection with every proper open subgroup of $G$.
\end{enumerate}
\end{ex}

\subsection{Decomposition rank of open subgroups}\label{sec:xi}

As in \S\ref{sec:drq}, we give another example of how the structure of compactly generated closed subgroups of a \tdlc group $G$ is controlled by the structure of compactly generated open subgroups of $G$.  Recall the class $\ms{E}$ of elementary groups and the decomposition rank $\xi$ defined in \S\ref{sec:elementary}.

Fix a group $G \in \ms{E}$ and let $K$ be a nontrivial closed compactly generated subgroup of $G$.  We have the following restrictions on $\xi(K)$:
\begin{enumerate}[(A)]
\item By \cite[Corollary~4.10]{Wesolek}, we have $\xi(K) \le \xi(G)$.
\item From the definition, we see that every value taken by $\xi$ is a successor ordinal.  We must also have $\xi(K) = \xi(\Res(K))+1$.  Thus $\xi(\Res(K)) = \alpha+1$ for some ordinal $\alpha$, and then $\xi(K) = \alpha+2$.
\end{enumerate}

We now show that in a very strong sense, the restrictions (A) and (B) are the only possible restrictions on $\xi(K)$, and there are no further restrictions on the ranks of compactly generated \emph{open} subgroups of $G$.

\begin{thm}\label{thm:open_xi}
Let $G$ be an elementary \tdlcsc group and let $\alpha$ be an ordinal such that $\alpha+2 \le \xi(G)$.  Then there is a compactly generated open subgroup $O$ of $G$ such that $\xi(O) = \alpha+2$.
\end{thm}

\begin{proof}
Write $\xi(G) = \pi+1$.  We have $G=\bigcup_{i\in \bN}O_i$ with $(O_i)_{i\in \bN}$ some (possibly terminating) increasing sequence of compactly generated open subgroups.  In this case
\[
		\pi=\sup_{i\in \bN}\xi(\Res(O_i)).
\]
In particular, since $\alpha < \pi$, there is some $i$ such that $\xi(\Res(O_i)) > \alpha$; it then follows that
\[
\xi(O_i) = \xi(\Res(O_i))+1 \ge \alpha+2.
\]
So certainly there is a compactly generated open subgroup of \emph{at least} the desired rank.

Now let $\beta$ be the least ordinal such that $\beta \ge \alpha$ and $G$ has a compactly generated open subgroup of rank $\beta+2$, say $\xi(O) = \beta+2$.  We may suppose for a contradiction that $\beta > \alpha$.  Then $\xi(\Res(O)) = \beta+1$ and $\alpha+2 \le \beta+1$.  Using the rank formula for $\Res(O)$, we see that $\Res(O)$ has a compactly generated open subgroup $K$  such that $\alpha+2 \le \xi(K) < \beta+2$; since restriction (B) above still applies, we have $\xi(K) = \gamma+2$ for some ordinal $\gamma$ such that $\alpha \le \gamma < \beta$.  Now in the chain
\[
K \le \Res(O) \le O \le G,
\]
each group is open or normal in the next.  By Theorem~\ref{thm:rio_closure_properties}(\ref{rio_closure_properties:2}), it follows that $K \in \mc{RIO}(G)$; indeed, since $K$ is compactly generated, $K \in \mc{IO}(G)$.  Applying Lemma~\ref{lem:IO_equi_gen}, there is a compactly generated open subgroup $E$ of $G$ such that $\Res(E) = \Res(K)$.  In particular,
\[
\xi(E) = \xi(\Res(E)) + 1 = \xi(\Res(K))+1 = \xi(K) = \gamma+2.
\]
This contradicts the minimality of $\beta$, proving that in fact $\beta = \alpha$, so $G$ has a compactly generated open subgroup $O$ of exactly the desired rank.
\end{proof}

\begin{rem}
The value of the ordinal $\alpha = \sup \{\xi(G) \mid G \in \ms{E}\}$ is not known at the time of writing: it could be $\omega_1$, or it could be some countable ordinal.  (In the latter case, it would be easy to show that there is actually an elementary group that achieves the largest possible rank.)  Theorem~\ref{thm:open_xi} at least shows (given that decomposition rank for infinite \tdlcsc groups is a commensurability invariant) that an elementary group of `large' rank necessarily has a `large' poset of commensurability classes of open subgroups.  Contrast with the group $G = \Aut(T_d)$ of automorphisms of a $d$-regular tree.  The \emph{closed} subgroup structure of $\Aut(T_d)$ is extremely rich: for example, given a compactly generated group $H$ without arbitrarily small nontrivial compact normal subgroups, then by taking a Cayley--Abels graph for $H$ and lifting to its covering tree, one can realize $H$ as a quotient of a closed subgroup of $\Aut(T_d)$ for all but finitely many $d$.  However, $\Aut(T_d)$ has only two commensurability classes of \emph{open} subgroups, namely the class of compact open subgroups, and the class that consists of $\Aut(T_d)$ and its unique open subgroup of index $2$.
\end{rem}

\section{The discrete residual}

\subsection{Action of the discrete residual}

Given an equicontinuously generated group $H$ of automorphisms of a \tdlc group $G$, the absence of $H$-invariant open subgroups restricts the way in which $\Res_G(H)$ can act on closed $H$-invariant subgroups of $G$.

\begin{prop}\label{prop:res_action}
Let $G$ be a \tdlc group, let $H \le \Aut(G)$ and let $R$ be an $H$-invariant closed subgroup of $G$ such that $\Res_R(H) = R$.  Let $K$ be a closed subgroup of $G$ that is both $R$-invariant and $H$-invariant.
\begin{enumerate}[(i)]
\item Let $L$ be a compactly generated open subgroup of $K$.  If $L$ is $H$-invariant, then $L$ is also $R$-invariant, and in fact
\[
[L,R]  \le \Res_L(R \rtimes H) \le \Res_{L,f}(R \rtimes H) = \Res_{L,f}(H).
\]
\item If $K$ is compact, then $[K,R] \le \Res_K(R \rtimes H) = \Res_K(H)$.
\item If $K$ is discrete, then $\CC_K(R)$ contains every finitely generated $H$-invariant subgroup of $K$.
\end{enumerate}
\end{prop}

\begin{proof}
Without loss of generality, we may replace $G$ with $\overline{RK}$ and assume that $K$ is normal in $G$.  Thus $K$ is a RIO subgroup of $G$; in particular, by Theorem~\ref{thmintro:RIO_separation}, $K$ is a RRS subgroup of $G$.

Let $L$ be a compactly generated $H$-invariant open subgroup of $K$.  Then by the RRS property, $L = K \cap O$ for some open subgroup $O$ of $G$.  In particular, $\N_G(L) \ge O$, so $\N_G(L)$ is open, and hence $\N_R(L)$ is an open subgroup of $R$.  We see that $\N_R(L)$ is also $H$-invariant.  Thus $\N_R(L) = R$.  Now let $M$ be an open $H$-invariant subgroup of $L$ of finite index.  Then $M$ is $R$-invariant by the same argument.  Thus $R$ acts on the finite discrete factor $L/M$ by conjugation; the kernel of this action, namely the set $\CC_R(L/M):= \{r \in R \mid \forall l \in L: [l,r] \in M\}$, is then an open subgroup of $R$.  Since $L$ and $M$ are $H$-invariant, $\CC_R(L/M)$ is also $H$-invariant, so $R = \CC_R(L/M)$.  In other words, $[L,R] \le M$.  As $M$ was an arbitrary open $H$-invariant subgroup of $L$ of finite index, we conclude that $[L,R] \le \Res_{L,f}(H)$.  In particular, every open $H$-invariant subgroup of $L$ of finite index is in fact $R \rtimes H$-invariant, so $\Res_{L,f}(H) = \Res_{L,f}(R \rtimes H)$.

Now let $N$ be an open subgroup of $L$ that is $R \rtimes H$-invariant, but not necessarily of finite index.  Then $P = \N_G(L) \cap \N_G(N)$ acts continuously on the discrete finitely generated group $L/N$; since $R \le P$, the action of $R$ on $L/N$ is also continuous.  In particular, every element of $L/N$ has open centralizer in $R$, and there is thus an open subgroup of $R$ that centralizes a generating set of $L/N$.  Then $\CC_R(L/N)$ is an open $H$-invariant subgroup of $R$, so $R = \CC_R(L/N)$; from the choice of $N$ we conclude that $[L,R] \le \Res_L(R\rtimes H)$.  Clearly $\Res_L(R\rtimes H) \le  \Res_{L,f}(R \rtimes H)$.  We have now proved all the assertions in (i).

Part (ii) is the special case of part (i) where $K = L$ and $K$ is compact, using the fact that every open subgroup of a compact group has finite index.

Part (iii) follows from the special case of part (i) where $K$ is discrete and $L$ is finitely generated, as in this case $\Res_L(R \rtimes H)$ is trivial, so $[L,R] = \triv$ and hence $L \le \CC_K(R)$.
\end{proof}

Consider now a compactly generated \tdlc group of the form $G = \Res_G(H) \rtimes H$.  We note that there is a certain compact normal subgroup of $G$ determined by the action of $H$ that naturally divides the action on compact normal subgroups into two parts.

\begin{prop}\label{prop:ergodic_nub}
Let $G$ be a \tdlc group, let $H \le \Aut(G)$ be compactly generated and let $R = \Res_G(H)$.
\begin{enumerate}[(i)]
\item There is a unique largest compact $H$-invariant subgroup $N$ of $G$ such that $H$ acts ergodically on $N$ and $R \le \N_G(N)$; moreover, $N \le R$ and $\N_G(N)$ is open.
\item Let $T/N$ be the elliptic part of the centre of $R/N$.  Then $T$ contains every compact normal $H$-invariant subgroup of $R$.
\end{enumerate}
\end{prop}

\begin{proof}
We see that $L = R \rtimes H$ is compactly generated by Theorem~\ref{thm:sr_envelope}(\ref{sr_envelope:1}).

Let $A$ be a compact $L$-invariant subgroup of $R$.  Then there is a compact open subgroup $U$ of $G$ such that $A \le U$.  By Proposition~\ref{sr_equivalents}, there is an open $H$-invariant subgroup $V$ of $G$ such that $V \subseteq UR$.  Let $L_V := V \rtimes H$; we see that every $L_V$-conjugate of $A$ is contained in $U$, so $A$ is contained in a compact $L_V$-invariant subgroup $A_2$ of $R$.  Now let $W$ be any open $H$-invariant subgroup of $G$ such that $W/R$ is compact and let $L_W := W \rtimes H$; then $A_2$ has only finitely many $L_W$-conjugates, each of which is normal in $R$, so $A_2$ is contained in a compact $L_W$-invariant subgroup $A_3$ of $R$.

From the preceding argument, we see that the union $C$ of all compact $L$-invariant subgroups of $R$ is in fact the union of all compact $L_W$-invariant subgroups of $R$, where $W$ is any open $H$-invariant subgroup of $G$ such that $W/R$ is compact.  By \cite[Theorem~3.3]{ReidWesolek}, there is a compact $L$-invariant subgroup $K$ of $R$ such that $C/K$ is discrete; in particular, $C$ is closed.  Set $N = \Res_{C}(L)$; then certainly $N \le K$, so $N$ is compact.  We see that in fact $N = \Res_C(L_W)$ by Lemma~\ref{lem:equi_residual}.  In particular, $W \le \N_G(N)$, so $\N_G(N)$ is open.  We see that $L$ acts ergodically on $N$ by Corollary~\ref{cor:res_ergodic}; it then follows by Proposition~\ref{prop:res_action}(ii) that in fact $\Res_N(H) = N$, so $H$ acts ergodically on $N$.

Now suppose $M$ is a compact $H$-invariant subgroup of $G$ on which $H$ acts ergodically, such that $R \le \N_G(M)$.  In other words, $M$ is a compact $L$-invariant subgroup of $G$ on which $H$ acts ergodically.  Then $H$ does not preserve any open subgroup of $M$, so $M \le U$ for every open $H$-invariant subgroup of $\N_G(M)$; thus $M \le R$.  It follows that $M \le C$, and hence 
\[
M = \Res_M(L) \le \Res_C(L) = N,
\]
proving (i).

From the construction of $N$, we see that for every compact normal $H$-invariant subgroup $S$ of $R$, the action of $H$ on $SN/N$ is residually discrete.  Applying Proposition~\ref{prop:res_action}, $C/N$ is a central factor of $R$.  Since $C/N$ is generated topologically by compact subgroups, we see that it is contained in the elliptic part of the centre of $R/N$, proving (ii).
\end{proof}

\subsection{Decomposition of the discrete residual}\label{sec:normal_series}

We can now prove Theorem~\ref{intro:normal_series}.

\begin{proof}[Proof of Theorem~\ref{intro:normal_series}]
By hypothesis, we have a \tdlc group $G \rtimes H$ such that $H$ is compactly generated.  Let $R = \Res_G(H)$.

We have already seen (Proposition~\ref{prop:residual_distal}) that $\Res_R(H) = R$.  The group $L := R \rtimes H$ is compactly generated by Theorem~\ref{thm:sr_envelope}(\ref{sr_envelope:1}).

By \cite[Theorem~1.3]{ReidWesolek}, there is a finite normal $H$-invariant series
\[
\triv = R_0 < R_1 < \dots < R_n = R,
\]
such that given $1 \le i \le n$, the factor $R_i/R_{i-1}$ is compact, discrete or a chief factor of $L$.

Suppose $R_i/R_{i-1}$ is a chief factor of $L$; that is, there does not exist a closed subgroup $S$ of $R$ such that $R_{i-1} < S < R_i$ and such that $S$ is both $H$-invariant and $R$-invariant.  Then it follows from Proposition~\ref{prop:res_action}(i) that $H$ does not leave invariant any proper compactly generated open subgroup of $R_i/R_{i-1}$; in particular, either $R_i/R_{i-1}$ is discrete or case (\ref{normal_series:4}) in the statement of the theorem is satisfied. 

Let us now consider the uppermost factor $R/R_{n-1}$.  Since $H$ does not leave invariant any proper open subgroup of $R$, we see that $R/R_{n-1}$ is not discrete.  If $R/R_{n-1}$ is compact then $H$ acts ergodically by Corollary~\ref{cor:res_ergodic}.  Otherwise, $R/R_{n-1}$ is a nondiscrete chief factor of $L$.  Thus $R/R_{n-1}$ satisfies at least one of cases (\ref{normal_series:3}) and (\ref{normal_series:4}) in the statement of the theorem.

It is natural to obtain the series in ascending order, in other words, $R_{i+1}$ is obtained by setting $R_{i+1} = N$ where $N/R_i$ is either a minimal $L$-invariant subgroup of $R/R_i$, or a sufficiently large compact or discrete subgroup of $R/R_i$.  The flexibility of the construction is such that whenever $R_{i+1}/R_i$ is compact, we can ensure $R_{i+1}/R_i$ is the largest compact $L$-invariant subgroup of $R/R_i$ up to finite index (such a subgroup exists by \cite[Theorem~3.3]{ReidWesolek}). Then $R_{i} \le N \le R_{i+1}$ where $N/R_i$ is the unique largest compact $L$-invariant subgroup of $R/R_i$ on which $H$ acts ergodically, as provided by Proposition~\ref{prop:ergodic_nub}(i).  By Proposition~\ref{prop:ergodic_nub}(ii), $R_{i+1}/N$ is a central factor of $R$.  We can thus ensure, possibly at the cost of increasing the length of the series, that every compact factor in the series satisfies (\ref{normal_series:3}) or is a central factor of $R$ (or both).  Thus every compact factor $R_i/R_{i-1}$ satisfies at least one of cases (\ref{normal_series:2}) and (\ref{normal_series:3}) in the statement of the theorem.

Now let $1 \le i < n$ and consider the factor $M = R_i/R_{i-1}$.  We have already dealt with the cases when $M$ is compact or a nondiscrete chief factor, so we may assume $M$ is discrete.  If $M$ is finitely generated, then $M$ is a central factor of $R$ by Proposition~\ref{prop:res_action}(iii), so case (\ref{normal_series:2}) of the theorem is satisfied.  Otherwise $M$ is infinitely generated and case (\ref{normal_series:1}) of the theorem is satisfied.  This exhausts the possibilities for $M$; thus we have obtained a finite normal $H$-invariant series for $R$ of the required form.
\end{proof}

In the statement of Theorem~\ref{intro:normal_series}, the uppermost factor $R/R_{n-1}$ of $R = \Res_G(H)$ cannot be discrete, and all of the finitely generated discrete factors occurring must be central in $R$.  The following example indicates why we cannot hope to restrict the discrete factors much further than this: for instance, we cannot require all discrete abelian factors to be central in $R$, nor can we avoid a discrete factor appearing as a quotient of a cocompact subgroup of $R$.

\begin{ex}
In $\mathrm{Sym}(\bZ)$, let $a_i = (3i \; 3i+1 \; 3i+2)$, let $b_i = (3i \; 3i+1)$ and let $c$ be the shift $n \mapsto n+3$.  Let $D$ be the group abstractly generated by $\{a_i \mid i \in \bZ\}$, let $U$ be the closure of the group generated by $\{b_i \mid b_i \in \bZ\}$ and let $S = \langle c \rangle$.  Then the group $G = DUS$ carries a unique \tdlc topology extending the permutation topology of $U$.  Let $H = \mathrm{Inn}(G)$.  Both $G$ and $H$ are compactly generated, and $R:= \Res_G(H) = DU$.  One sees that in any finite series for $R \rtimes H$ of the form specified by Theorem~\ref{intro:normal_series}, there will be an infinitely generated discrete abelian factor $R_{i+1}/R_{i}$ appearing in the series such that $R_{i+1}$ is cocompact in $R$ and $R_{i+1}/R_i$ is not a central factor of $R$.
\end{ex}

\appendix

\section{Metrizable approximation of group actions}

In an earlier version of this article, the main theorem was proved by first reducing to actions on metrizable spaces.  This reduction turned out to be unnecessary, but is perhaps of independent interest; we preserve the arguments in this appendix.

Let $G$ be a \tdlc group, let $H \le \Aut(G)$ and let $K$ be a closed $H$-invariant subgroup.  As long as $H$ is not too large ($\sigma$-equicontinuous is a sufficient constraint), then $X = G/K$ can be approximated as a topological $H$-space by completely metrizable $H$-spaces.

\begin{defn}
A function $f: X \rightarrow Y$ between topological spaces is \defbold{proper} if it is continuous and the preimage of any compact set is compact.

Given equivalence relations $R_1,R_2$ on a set $X$, say $R_1 \le R_2$ if $R_1 \subseteq R_2$, regarded as subsets of $X \times X$.

Let $H$ be a group acting by homeomorphisms on a locally compact Hausdorff space $X$.  The $H$-space $X$ is \defbold{pro-metrizable} (with filtration $(\sim_i)_{i \in I}$) if there is a descending net $(\sim_i)_{i \in I}$ of equivalence relations, each preserved by $H$, such that:
\begin{enumerate}[(a)]
\item If $x \sim_i y$ for all $i \in I$, then $x = y$.
\item For each $i \in I$, the quotient space $X/\!\sim_i$ is locally compact and completely metrizable and the quotient map $p_i: X \rightarrow X/\!\sim_i$ is proper.
\end{enumerate}
\end{defn}

The following lemma shows that the structure of $X$ is approximated well by that of the quotients $X/\!\sim_i$.

\begin{lem}
Let $X$ be a pro-metrizable space with filtration $(\sim_i)_{i \in I}$, and let $p: X \rightarrow \prod_{i \in I}X/\!\sim_i$ be the product of the maps $p_i$.  Then $p$ is a proper topological embedding.
\end{lem}

\begin{proof}
By the standard characterization of the product topology, $p$ is continuous.  In particular, the image of every compact subset of $X$ is compact, and hence closed.  Conversely, given a compact subset $K$ of $\prod_{i \in I}X/\!\sim_i$, then the projection $\pi_i(K)$ of $K$ onto each of the coordinates is compact, so $p\inv(K) \subseteq p\inv(\pi_i(K)) = p_i\inv(K)$ for each $i$.  Since $p_i$ is proper, $p_i\inv(K)$ is compact, so $p\inv(K)$ is relatively compact; by continuity $p\inv(K)$ is closed, hence compact.  Thus $p$ is proper.

We see from condition (a) of the definition of a pro-metrizable space that $p$ is injective.  Thus $p$ induces a one-to-one correspondence between compact subsets of $X$ and those of $p(X)$.  The topology of $X$ is generated by relatively compact open sets; in turn, every relatively compact open set $O$ is the difference of two compact sets (namely, $O = \overline{O} \smallsetminus (\overline{O} \smallsetminus O)$).  Thus every open set in $X$ has open image in $p(X)$, in other words, $p$ is a topological embedding.
\end{proof}

Under some fairly general circumstances, coset spaces give pro-metrizable actions in this sense.

\begin{prop}\label{pro-metrizable:coset}
Let $G$ be a \tdlc group, let $H$ be a $\sigma$-equicontinuous group of automorphisms of $G$ and let $K$ be a closed $H$-invariant subgroup.  Then $G/K$ is pro-metrizable as an $H$-space, with completely metrizable quotients of the form $Z_i = N_i \bs G / K$ where $(N_i)_{i \in I}$ is a descending net of compact $H$-invariant subgroups with trivial intersection.
\end{prop}

\begin{proof}
Let us note first that it follows from Urysohn's metrization theorem (see for instance \cite[Corollary 23.2]{Willard}) that every Hausdorff quotient space of a compact metrizable space is completely metrizable.

Let $H$ be the union of an ascending sequence $(S_n)$ of equicontinuous sets such that $1 \in S_n = S\inv_n$ for all $n \in \bN$.

Fix a descending net $(U_i)_{i \in I}$ of compact open subgroups of $G$ with trivial intersection.  Given $i \in I$, let $U_i(n) = \bigcap_{s \in S_n}s(U_i)$, let $V_i(n)$ be the core of $U_i(n)$ in $U_i$, let $M_i = \bigcap^\infty_{n=1}V_i(n)$ and let $N_i = \bigcap^\infty_{n=1}U_i(n)$.  Then $M_i$ is a compact normal subgroup of $U_i$; by construction, we see that $U_i/M_i$ is a first-countable, hence completely metrizable, profinite group.  Since $M_i \le N_i$, it follows that the right coset space $N_i \bs U_i$ is completely metrizable.

Now define $Z_i = N_i \bs G / K$; the corresponding equivalence relation $\sim_i$ on $G/K$ is then given by $xK \sim_i yK$ if $N_ixK = N_iyK$.  Since both $N_i$ and $K$ are $H$-invariant subgroups of $G$, we see that $\sim_i$ is an $H$-invariant equivalence relation, so $H$ acts on $Z_i$.  We see that $\{U_ixK/K \mid i \in I\}$ is a base of neighbourhoods of $xK$ for any given $x \in G$.  In particular, it is now clear that given $x,y \in G$, if $xK \sim_i yK$ for all $i \in I$ then $xK = yK$.  The relation $\sim_i$ is the orbit relation of the compact group $N_i$ acting on $G/K$, so the quotient map from $G/K$ to $Z_i$ is proper.

It remains to show that $Z_i$ is completely metrizable.  Fix $i \in I$ and write $N = N_i$ and $U = U_i$.  We can partition $Z_i$ into clopen sets $O_x$, where $O_x = \{NyK \mid y \in G, NyK \subseteq UxK\}$.  Now fix $x \in G$; it suffices to show $O_x$ is completely metrizable.  We see that $NyK \in O_x$ if and only if $yK = zxK$ for some $z \in U$.  There is therefore a surjective map $\theta$ from $U$ to $O_x$ given by $z \mapsto NzxK$.  This map is easily seen to be continuous; since $U$ is compact, in fact $\theta$ is closed, so $\theta$ is a quotient map.  We have $\theta(a) = \theta(b)$ if and only if $b \in Na(xKx\inv \cap U)$, so $O_x$ is homeomorphic to the double coset space $N \bs U / R$ where $R = xKx\inv \cap U$.  Since $N$ and $R$ are closed subgroups of the compact group $U$, the double coset space is Hausdorff; hence $N \bs U / R$ is completely metrizable.  Thus $O_x$ is completely metrizable, and hence $Z_i$ is completely metrizable.
\end{proof}

Here is a condition under which a distal fixed point remains distal in a quotient space.

\begin{lem}\label{distal_compact_extension}
Let $H$ be a group acting by homeomorphisms on a topological space $X$.  Let $\sim$ be an $H$-invariant equivalence relation, such that the quotient map $p: X \rightarrow X/\!\sim$ is proper and $X/\!\sim$ is a locally compact Hausdorff space.  Suppose that $z \in X$ is fixed by $H$ and that $[z]$ consists of distal points for the action of $H$ on $X$.  Then $[z]$ is a distal point of the action of $H$ on $X/\!\sim$.
\end{lem}

\begin{proof}
Given $x \in X$, write $[x]$ for the $\sim$-class of $x$.  Suppose that $([x],[z])$ is a proximal pair for the action of $H$ on $X/\!\sim$: that is, there exists a net $(h_i)$ in $H$ such that $h_i([x]) \rightarrow [z]$.  Eventually $h_i([x])$ will be confined to a compact neighbourhood $K$ of $[z]$; since $p$ is proper, it follows that $h_i(x)$ is confined to a compact neighbourhood $p\inv(K)$ of $z$.  Thus by passing to a subnet, we may assume $h_i(x)$ converges, say $h_i(x) \rightarrow w \in [z]$.

Consider $[z]$ as a subspace of $X$; then $[z]$ is compact, since $p$ is proper.  Moreover, $H[z] = [z]$, since $Hz = z$ and $\sim$ is $H$-invariant, and by hypothesis $[z]$ consists of distal points for the action of $H$ on $X$.  Thus $x \in [z]$ by Lemma~\ref{lem:distal_orbit_closure}.  In particular, $[x] = [z]$, proving that $[z]$ is a distal point of the action of $H$ on $X/\!\sim$.
\end{proof}

We conclude this section with a decomposition for actions with an N-distal fixed point on pro-metrizable spaces.  As noted in Lemma~\ref{lem:distal_equiv}, the singleton of a distal fixed point is always an intersection of open invariant neighbourhoods; what is significant here is that we obtain a supply of proper \emph{closed} invariant neighbourhoods.

\begin{defn}Let $G$ be a group acting on a topological space $X$.  The action has \defbold{generic neighbourhoods at $x \in X$} if every $G$-invariant neighbourhood of $x$ is dense; in other words, for all $y \in X$, there are nets $(x_i)$ in $X$ and $g_i$ in $G$ such that $x_i \rightarrow x$ and $g_ix_i \rightarrow y$.\end{defn}

Note that the action is topologically transitive, that is, every nonempty open $G$-invariant set is dense, if and only if the action has generic neighbourhoods at every point.  An action can have a generic neighbourhood point without being topologically transitive: for example, the action of the multiplicative group $\bR_{>0}$ of positive reals on $\bC$ via $r(z) = rz$ is not topologically transitive, but $z=0$ is a generic neighbourhood point.  However, for distal actions on compact metrizable spaces, the existence of a generic neighbourhood point is equivalent to minimality of the action.

\begin{lem}\label{lem:top_trans:metrizable}Let $H$ be a group acting by homeomorphisms on a completely metrizable space $X$.  Then $x \in X$ has generic neighbourhoods if and only if there is a dense $G_{\delta}$ set of points $D$ such that $x \in \overline{Hy}$ for all $y \in D$.\end{lem}

\begin{proof}
Suppose $x$  has generic neighbourhoods under the action and fix a complete metric $d$ on $X$ compatible with the topology.  Let $\mc{O}$ be the set of open $H$-invariant neighbourhoods of $x$.  For each positive integer $n$, let $O_n = \bigcup_{g \in G}g(B_n)$, where $B_n$ is the open ball of radius $1/n$ around $x$.  Then $O_n$ is an $H$-invariant neighbourhood of $x$.  By hypothesis, it follows that $O_n$ is dense in $X$.  Moreover, we see that for all $O \in \mc{O}$, there exists $n$ such that $B_n \subseteq O$, so that in fact $O_n \subseteq O$.  Thus $\bigcap_{O \in \mc{O}}O = \bigcap^\infty_{n=1}O_n$.  By the Baire Category Theorem, we conclude that $D = \bigcap_{O \in \mc{O}}O$ is a dense $G_{\delta}$ set.

Let $y \in D$.  Then the set $V = X \smallsetminus \overline{Hy}$ is an open $H$-invariant set that does not contain $y$.  Since $y \in O$ for all $O \in \mc{O}$, it follows that $x \not\in V$, in other words $x \in \overline{Hy}$.

Conversely, suppose there is a dense set of points $D$ such that $x \in \overline{Hy}$ for all $y \in D$.  Let $O$ be an open $H$-invariant neighbourhood of $x$ and suppose $z \in X \smallsetminus \overline{O}$.  Then $\overline{Hz}$ is contained in $X \smallsetminus O$, so $x \not\in \overline{Hz}$ and hence $z \not\in D$.  Thus $X \smallsetminus \overline{O}$ is an open set disjoint from $D$.  Since $D$ is dense, we conclude that $X \smallsetminus \overline{O} = \emptyset$, in other words $O$ is dense.
\end{proof}

\begin{cor}\label{cor:distal_generic}
Let $H$ be a group acting by homeomorphisms on a completely metrizable space $X$.  Suppose $x$ has generic neighbourhoods and that $\overline{Hx}$ is compact and consists of distal points for the action.  Then $X = \overline{Hx}$ and the action of $H$ on $X$ is minimal.
\end{cor}

\begin{proof}
By Lemma~\ref{lem:top_trans:metrizable}, there is a dense $G_{\delta}$ set of points $D$ such that $x \in \overline{Hy}$ for all $y \in D$.  By Lemma~\ref{lem:distal_orbit_closure}, $H$ acts minimally on $\overline{Hx}$ and we have $\overline{Hy} = \overline{Hx}$ for all $y \in D$.  In particular, $\overline{Hx}$ is dense, and hence equal to $X$.
\end{proof}

\begin{prop}\label{prop:metrizable_transfinite}
Let $H$ be a group and let $X$ be an $H$-space that is either pro-metrizable or completely metrizable and let $x \in X$ be a fixed point of $H$.  For each closed subset $Y$ of $X$, let $\mc{K}_Y$ be the set of closed $H$-invariant neighbourhoods of $x$ in $Y$.  Let $X^0 = X$, for each ordinal $\alpha$ let $X^{\alpha+1} = \bigcap_{K \in \mc{K}_{X^\alpha}}K$ and for each limit ordinal $\lambda$ let $X^{\lambda} = \bigcap_{\alpha < \lambda}X^{\alpha}$.  Let $X^{\infty}$ be the intersection of the spaces $X^{\alpha}$ as $\alpha$ ranges over the ordinals.

Then the following hold:
\begin{enumerate}[(i)]
\item We have $X^{\infty} = X^{\alpha}$, where $\alpha$ is minimal such that $X^{\alpha} = X^{\alpha+1}$.  The sets $(X_{\beta})_{\beta \le \alpha}$ then form a strictly descending sequence of closed $H$-invariant subspaces of $X$.
\item The subspace $X^{\infty}$ is the unique largest subspace of $X$ such that $x$ has generic neighbourhoods under the action of $H$ on $X^{\infty}$.
\item Suppose that there is an open neighbourhood $O$ of $x$ such that every compact $H$-invariant subset of $O$ consists of distal points for the action of $H$ on $X$.  Then $X^{\infty} = \{x\}$.
\item If $X$ is completely metrizable, there is a countable subset $\mc{L}$ of $\mc{K}_X$ such that $\bigcap_{K \in \mc{L}}K = \bigcap_{K \in \mc{K}_X}K$.
\end{enumerate}
\end{prop}

\begin{proof}
It is clear that we produce a (weakly) descending sequence of closed $H$-invariant subspaces of $X$, so $X^{\alpha} = X^{\alpha+1}$ for some $\alpha$; let $\alpha$ be the first ordinal for which this occurs.  Then we see that every closed $H$-invariant neighbourhood of $x$ in $X^{\alpha}$ is dense.  It follows by induction that $X^{\alpha} = X^{\beta}$ for all $\beta > \alpha$, so $X^{\alpha} = X^{\infty}$.  By the choice of $\alpha$, the sets $(X_{\beta})_{\beta \le \alpha}$ then form a strictly descending sequence, proving (i).

Suppose $Y$ is an $H$-invariant subspace of $X$ in which $x$ has generic neighbourhoods.  Then $Y \subseteq X^0$, and whenever $Y \subseteq X^{\alpha}$ we see that $Y$ is contained in every closed $H$-invariant neighbourhood of $x$ in $X^{\alpha}$, so $Y \subseteq X^{\alpha+1}$.  By induction, $Y \subseteq X^{\infty}$.  Conversely, since $X^{\infty} = X^{\alpha} = X^{\alpha+1}$ for some $\alpha$, we see that there are no proper closed $H$-invariant neighbourhoods of $x$ in $X^{\infty}$, so $x$ is a generic neighbourhood point for the action on $X^{\infty}$.  This proves (ii).

Now suppose that there is an open neighbourhood $O$ of $x$ such that every compact $H$-invariant subset of $O$ consists of distal points for the action of $H$ on $X$.  Suppose $X$ is pro-metrizable with filtration $(\sim_i)_{i \in I}$.  Letting $i$ range over $I$, the sets $([x]_i \smallsetminus O)$ form a descending net of compact sets with empty intersection: thus there exists $i \in I$ such that $[x]_i \subseteq O$.  So without loss of generality, we may assume $[x]_i \subseteq O$ for all $i \in I$.  In particular, $[x]_i$ is a compact $H$-invariant set contained in $O$, so $[x]_i$ consists of distal points for the action.  Fix $i \in I$, let $Y = X^{\infty}/\!\sim_i$ and let $p: X \rightarrow Y$ be the quotient map.  By Lemma~\ref{distal_compact_extension}, $y:= [x]_i$ is a distal fixed point of $H$, but at the same time $y$ is a generic neighbourhood point of $H$ on $Y$.  Since $Y$ is a completely metrizable space, we conclude that $Y = \{y\}$ by Corollary~\ref{cor:distal_generic}.  Letting $i$ range over $I$, we see that each of the spaces $X^{\infty}/\!\sim_i$ has a single point, so in fact $X^{\infty}$ only has a single point.  Similarly, if $X$ itself is completely metrizable, then $X^{\infty}$ is completely metrizable and we again have $X^{\infty} = \{x\}$ by Corollary~\ref{cor:distal_generic}.  This completes the proof of (iii).

If $X$ is completely metrizable, then we see that every $K \in \mc{K}_X$ contains $\overline{O_n}$ for some $n$, where $O_n$ is as in the proof of Lemma~\ref{lem:top_trans:metrizable}, and hence it suffices to take $\mc{L} = \{\overline{O_n} \mid n \in \bN\}$, proving (iv).
\end{proof}

\end{document}